\documentclass[10pt]{amsart}
\usepackage[hmarginratio=1:1]{geometry}
\usepackage{amsmath, amsfonts,amssymb,amsthm,amscd,latexsym,euscript,xypic, verbatim}
\usepackage{amstext}
\usepackage{amsmath}
\usepackage{amssymb}
\usepackage{graphicx}
\usepackage{algorithm}
\usepackage[active]{srcltx}
\usepackage{graphics}
\usepackage{color}
\usepackage[utf8]{inputenc}
\usepackage[Lenny]{fncychap}
\usepackage[all]{xy}
\usepackage[colorlinks=false]{hyperref}
\usepackage[english]{babel}
\usepackage{mathtools}
\usepackage{bbm}
\usepackage[noend]{algpseudocode}

\makeatletter
\def\BState{\State\hskip-\ALG@thistlm}
\makeatother

\usepackage{tikz}
\usetikzlibrary{matrix,chains,positioning,decorations.pathreplacing,arrows}
\usetikzlibrary{patterns}
\usetikzlibrary{matrix}
\usetikzlibrary{arrows,shapes}

\theoremstyle{plain}

\newtheorem*{thm1.2}{(1.2) Theorem}
\newtheorem*{thm1.3}{(1.3) Theorem}
\newtheorem*{thm1.4}{(1.4) Theorem}
\newtheorem*{propA*}{Proposition A}
\newtheorem*{propB*}{Proposition B}
\newtheorem*{thmC*}{Theorem C}
\newtheorem*{propD*}{Proposition D}
\newtheorem{prop}{Proposition}[section]
\newtheorem{thm}[prop]{Theorem}
\newtheorem{problem}[prop]{Problem}
\newtheorem{conj}{Conjecture}
\newtheorem{cor}[prop]{Corollary}
\newtheorem{lemma}[prop]{Lemma}

\theoremstyle{definition}
\newtheorem{point}[prop]{}

\newtheorem{Def}[prop]{Definition}

\newtheorem*{Def*}{Definition}

\newtheorem{example}[prop]{Example}
\newtheorem*{example*}{Example}

\newtheorem*{notation*}{Notation}
\newtheorem*{question*}{Question}
\newtheorem{remark}[prop]{Remark}

\newtheorem*{rem}{Remark}

\newtheorem{innercustomthm}{Theorem}
\newenvironment{customthm}[1]
{\renewcommand\theinnercustomthm{#1}\innercustomthm}
{\endinnercustomthm}

\newcommand{\HH}{\mathcal H}

\newcommand{\Oliver}[1]{{\color{orange}Oliver: #1}}
\title{Computational Complexity of learning algebraic varieties}
\author{Oliver G\"afvert}
\address{Department of mathematics, KTH, 10044,
	Stockholm, Sweden}
\email{oliverg@math.kth.se}
\urladdr{https://people.kth.se/~oliverg/}

\begin{document}
	\maketitle

\begin{abstract}
	We analyze the complexity of fitting a variety, coming from a class of varieties, to a configuration of points in $\Bbb C^n$. The complexity measure, called the \textit{algebraic complexity}, computes the Euclidean Distance Degree (EDdegree) of a certain variety called the \textit{hypothesis variety} as the number of points in the configuration increases. Finally, we establish a connection to complexity of architectures of polynomial neural networks.

%
%
	
	For the problem of fitting an $(n-1)$-sphere to a configuration of $m$ points in $\Bbb C^n$, we give a closed formula of the algebraic complexity of the hypothesis variety as $m$ grows for the case of $n=1$. For the case $n>1$ we conjecture a generalization of this formula supported by numerical experiments.
\end{abstract}

\section{Introduction}

A fundamental problem in data analysis is recovering model parameters from noisy measurements.
This is a basic problem in manifold learning/dimensionality reduction \cite{Lee:2007:NDR:1557216}, statistical regression and learning theory \cite{hastie01statisticallearning, mlbook}. We phrase this problem in the setting of algebraic geometry and use tools coming from this field to study it. The model is in our case an algebraic variety $V$, the parameters are coefficients of polynomials defining an ideal cutting out $V$ and the measurements are points sampled from $V$ with noise from a specified distribution. 


To recover the unknown variety $V$ we assume that the points $p_1, p_2, \dots, p_m$ are sampled from $V$ with Gaussian noise and that $V$ comes from a class of varieties $\HH$. We then look for the varieties lying closest to the points in the sense that small perturbations lie on a variety in $\HH$. If $\HH$ is the class of all hypersurfaces of degree $\leq d$, this is called \textit{polynomial regression}, which is a special case of \textit{linear regression} \cite{mlbook}. In this paper we consider more general classes of varieties and therefore use the broader name \textit{algebraic regression}. The goal of the paper is to analyze the computational complexity of finding the variety in a class $\HH$ that best fit a given set of samples. To this end, we develop a complexity measure, called the \textit{algebraic complexity of $\HH$}, based on the Euclidean Distance Degree (EDdegree) \cite{edd}.


The class of varieties $\HH$ is in the theory of \textit{Probably Approximately Correct (PAC) learning} \cite{mlbook} called a \textit{hypothesis class}. The two fundamental invariants of a hypothesis class are the \textit{sample complexity} and the \textit{computational complexity}. The sample complexity tells you how many points you need to sample in order to recover the variety with some probability. It measures the richness, or expressibility of a hypothesis class, while the computational complexity measures the complexity of implementing the learning rule, which in our case means solving an optimization problem. It tells you the amount of work you need to perform in order to obtain an \textit{optimal} hypothesis. To analyze the sample complexity of $\HH$, one may use tools such as the \textit{Vapnik–Chervonenkis (VC) dimension} and the \textit{Rademacher complexity} \cite{mlbook}. Analyzing the computational complexity is much harder as many learning problems are NP-hard to compute \cite{mlbook}. It would be desirable to characterize their computational complexity relative to each other. We propose using the Euclidean Distance Degree (EDdegree) to analyze the computational complexity. The EDdegree measures the \textit{algebraic} complexity of a polynomial optimization problem and we will show in this paper how we can use it to say something about the computational complexity of regression problems. 

In our setting, the hypothesis class $\HH$ is defined by polynomial equations and is thus a variety. We therefore refer to it as the \textit{hypothesis variety}. Rather than measuring expressibility of $\HH$, like the VC-dimension, the EDdegree instead characterizes the complexity of finding an \textit{optimal} hypothesis in $\HH$ for a learning rule, given a set of samples.  The EDdegree is the degree of the polynomial describing the \textit{optimal} solutions to the regression problem of fitting a function from $\HH$ to the given samples. By optimal, we mean local minimum/maximum or saddle points of an optimization problem (see Section \ref{sec:3}), these are called the \textit{critical points}. The number of critical points is dependent on the number of samples and therefore we consider the growth of the EDdegree of $\HH$ as the number of samples grow. We call this function the \textit{algebraic complexity} of $\HH$. The following example illustrates the meaning of the critical configurations of a specific point configuration. The varieties passing through the critical configurations are the varieties that best fit the configuration.
\begin{example}\label{ex:circle}
	Consider a configuration of four points in $\Bbb C^2$.  The following show four circles passing through \textit{real critical configurations} of Problem (\ref{eq:prob}), out of a total of 26 critical configurations:
	\begin{center}
		\begin{figure}[H]
			\begin{tabular}{llll}
				\includegraphics[width=0.23\linewidth]{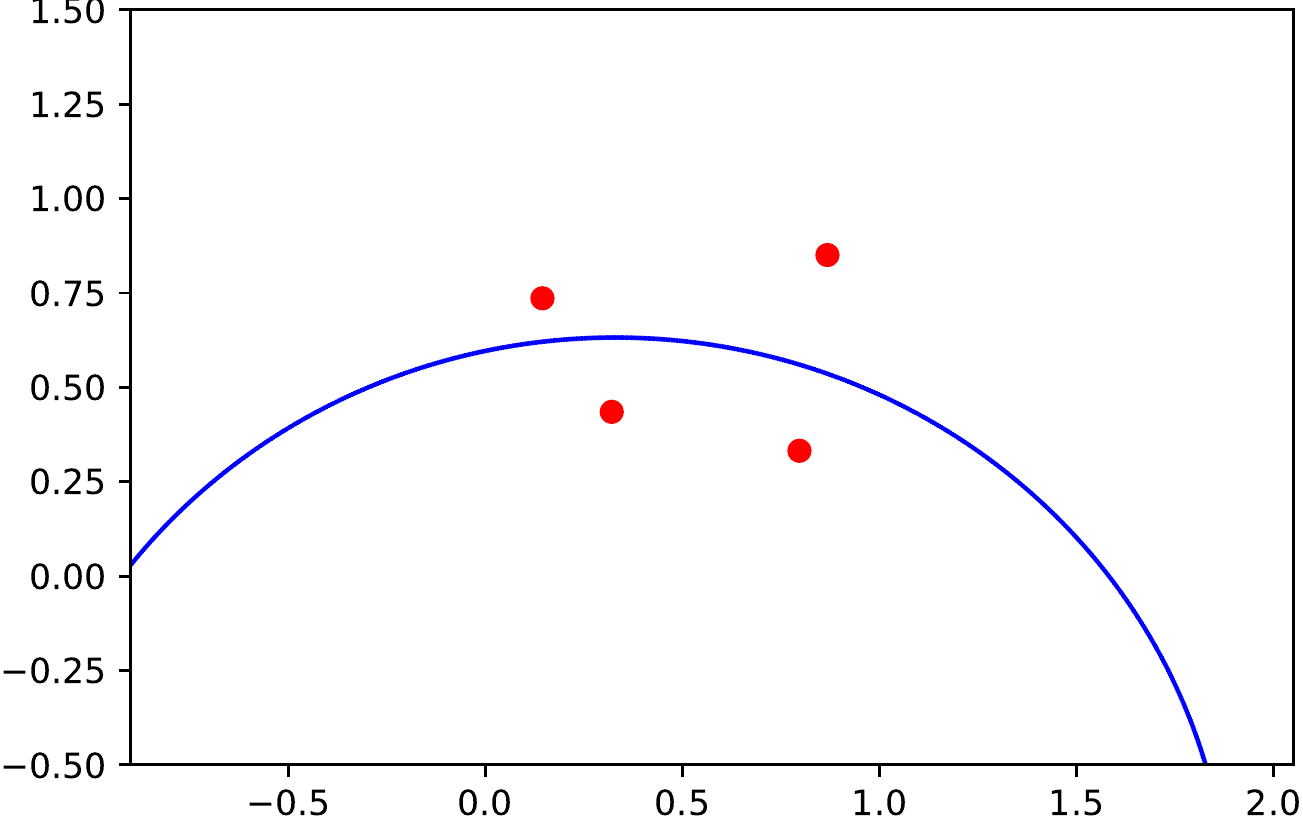}
				&
				\includegraphics[width=0.23\linewidth]{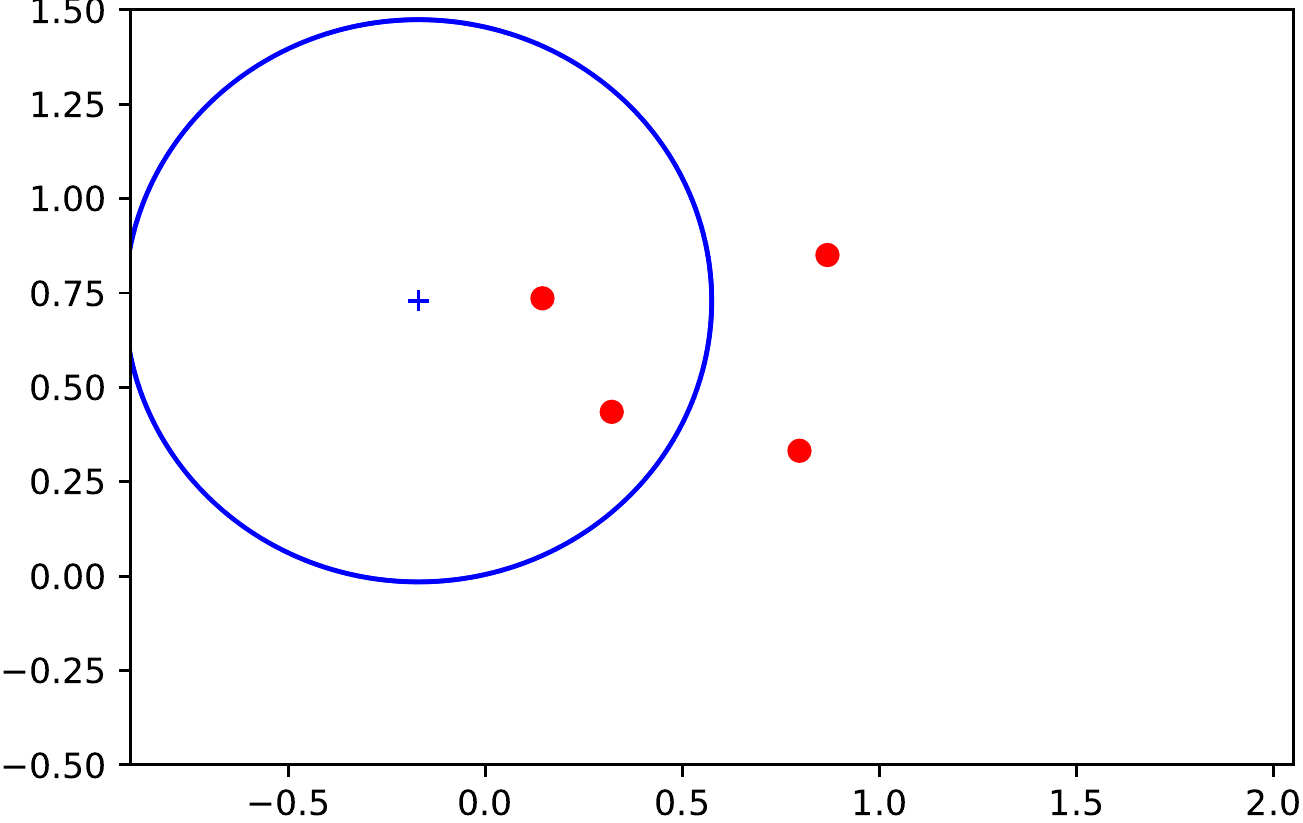}&
				\includegraphics[width=0.23\linewidth]{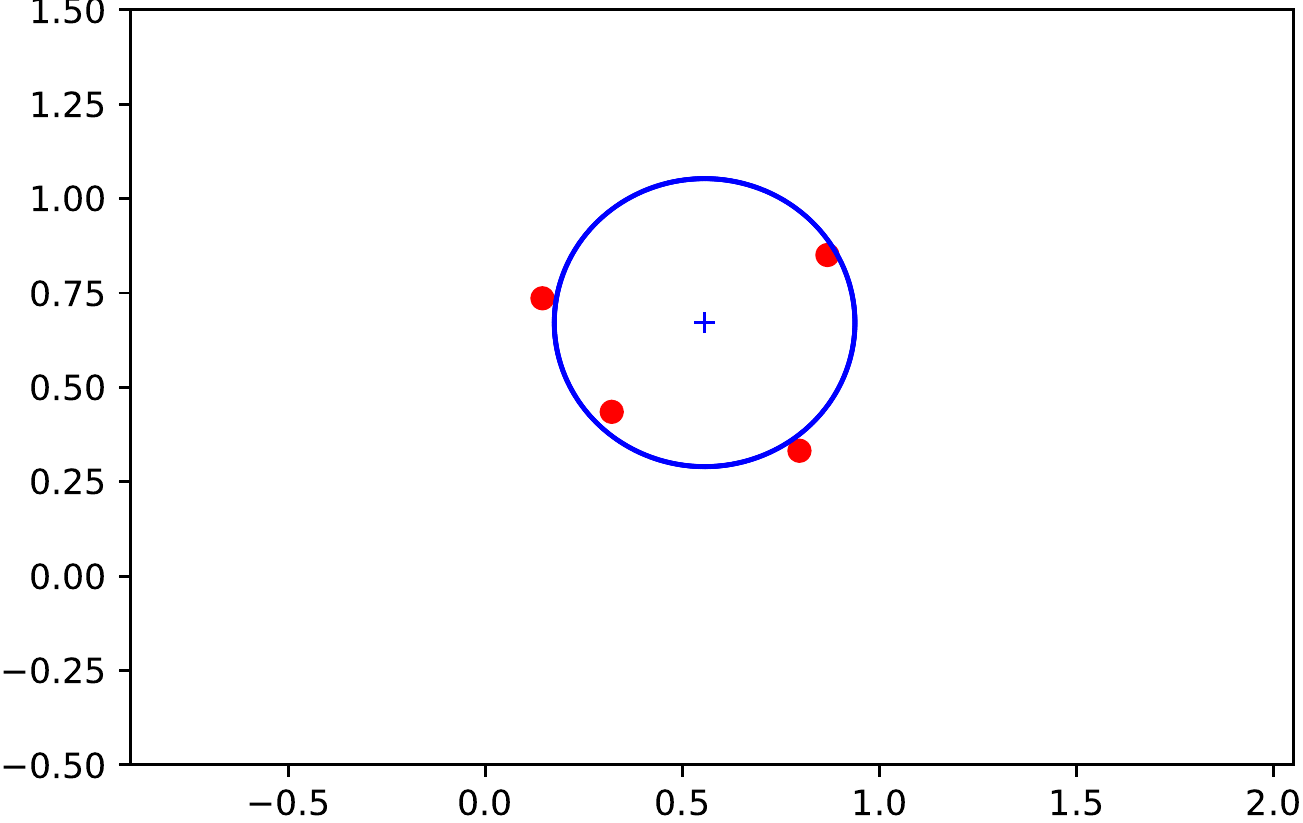}
				&
				\includegraphics[width=0.23\linewidth]{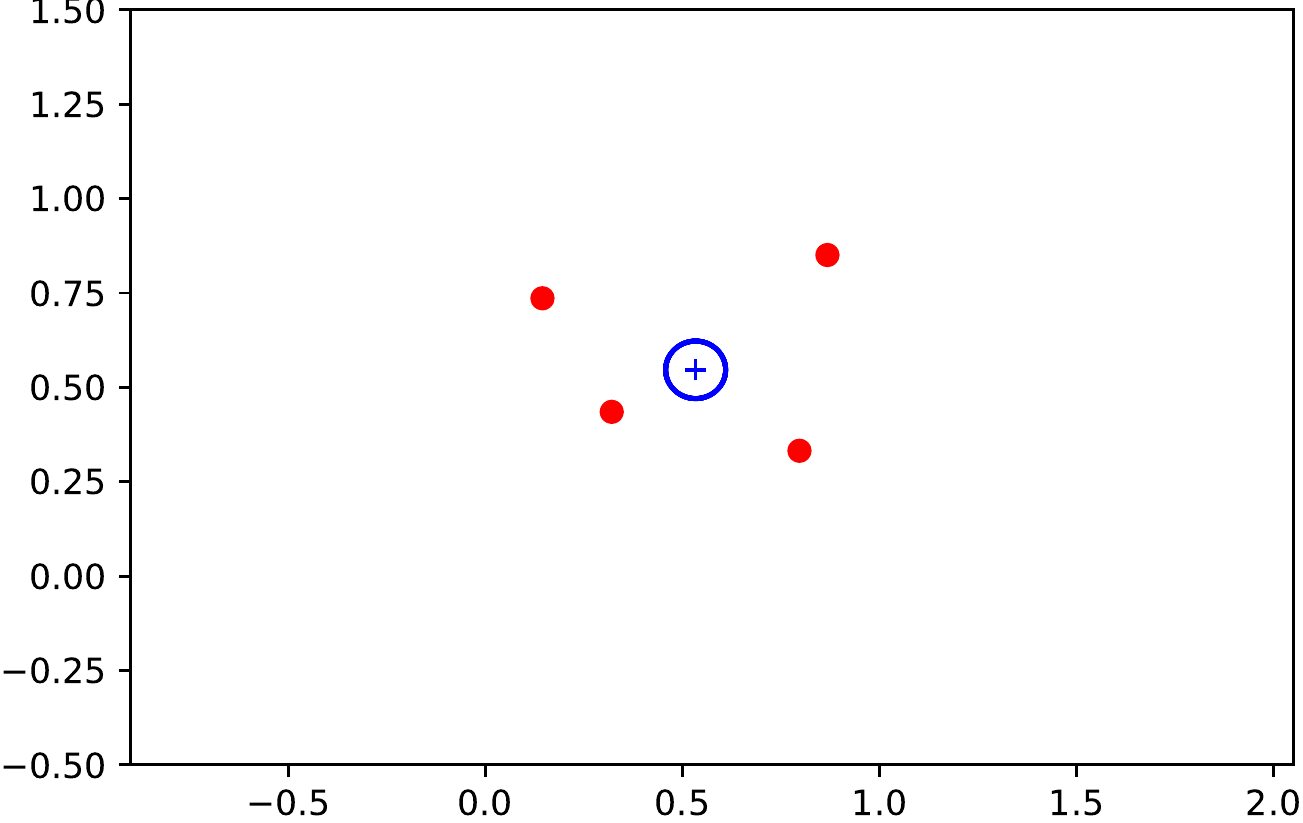}
			\end{tabular}
			\caption{Four of the critical circles of the point configuration.}
		\end{figure}
	\end{center}
\end{example}
The procedure in the above example provides a non-linear generalization of \textit{principal component analysis (PCA)} (see Section \ref{sec:edd}). For the hyperplane class we get the usual (linear) principal components while for other choices of $\HH$ we get something else. For the non-linear hypothesis varieties we consider in this paper, the number of principal components we get is dependent on the number of points in the configuration.

In this paper we investigate the algebraic complexity of a series of classes of varieties. Starting from the simplest classes, the class of hyperplanes in $\Bbb C^n$ and the class of affine subspaces of codimension $\leq r$, both of which have closed form expressions for their algebraic complexity. We then continue with analyzing the class of $(n-1)-$spheres, for which we prove the following closed formula for the algebraic complexity in the case $n=1$:\begin{customthm}{1}\label{thm:n1}
	$$EDdegree(\HH_{m, 1}) = 2^{m-1}-1$$
	and if $p$ is a real configuration, then the critical points of $\HH_{m, 1}$ with respect to $p$ are all real.
\end{customthm} For the general case, $n\geq 1$, we conjecture the following formula, based on the numerical numerical experiments in Table \ref{tab:exp}:
\begin{conj} \label{conj:circles}
	The number of critical $(n-1)-$spheres of a generic configuration of $m$ points in $\Bbb C^n$, where $m>n+2$, is given by: 
	$$EDdegree(\HH_{m, n}) = EDdegree(\HH_{m-1, n}) +\sum_{k=0}^{m-1} { m-1 \choose k }\sum_{l=0}^n{{k+1} \choose l}- 2^{m-1} - m2^{m-2}$$
\end{conj} Finally we end by considering paraboloids and ellipsoids for which we provide numerical results on their algebraic complexity. The \textbf{main contributions} of the paper are the following:
\begin{itemize}
	\item defining the algebraic complexity (Definition \ref{def:hypvar}) as a complexity measure for algebraic regression problems.
	\item developing tools, such as a weighted EDdegree called $\beta EDdegree$, for computing and approximating the algebraic complexity (Theorem \ref{thm:ineq}, Corollary \ref{prop:cond}, Corollary \ref{cor:crit} and Proposition \ref{prop:ci}).
	\item proving and conjecturing closed formulas of the algebraic complexity for prescribed classes of varieties (Theorem \ref{thm:n1} and Conjecture \ref{conj:circles}) and relating the algebraic complexity to the generalized Eckart-Young theorem (Theorem \ref{thm:gene}).
	\item establishing a connection between the algebraic complexity of the prescribed classes in Section \ref{sec:4} and architectures of polynomial neural networks (see Section \ref{sec:nn}). 
\end{itemize}
The computation of the algebraic complexity relies on computing the EDdegree. We compute the EDdegree directly from the defining equations of the \textit{critical ideal} (see Section \ref{sec:edd}) using numerical methods. Other methods are possible, such as computing it using Chern classes, polar classes and their extensions to singular varieties \cite{edd, ZHANG201855} or the Euler characteristic and Euler obstruction function \cite{aluffi2018, 2018arXiv181205648M, 2019arXiv190105550M}. For small examples we compute the EDdegree using Macaulay2 \cite{M2} and for all other examples we use Bertini \cite{BHSW06} and the method of regeneration \cite{10.2307/41104703}.  Another alternative is to use the method by Martín del Campo and Rodriguez \cite{MARTINDELCAMPO2017559} using monodromy loops to compute the EDdegree, but this we leave for future work.

\subsection{Future work}
For future work we would like to investigate more complicated classes of varieties such as polynomial neural networks, which are neural networks \cite{mlbook} with polynomial activation functions. A special type of polynomial neural network was recently studied by Kileel et. al. in the paper \cite{2019arXiv190512207K} where they compute the dimension of $C_{1, n}$ (see Section \ref{sec:3}) for these types of networks. They also note that the EDdegree, and by extension the algebraic complexity, could be useful to characterize the complexity of different architectures of these networks and polynomial neural networks in general. The algebraic complexity measures the number of critical points of the Euclidean distance function, which in the setting of polynomial neural networks is the same as the \textit{mean-squared-error (MSE)} objective function, which is one of standard objective functions used in practice.

\subsection{Organization} The paper is organized as follows. In Section \ref{sec:2} we review background, such as multivariate polynomial interpolation, the Euclidean distance degree and define the relevant terms we will use. In Section \ref{sec:3} we define the \textit{hypothesis variety}, which is our main object of study, and the algebraic complexity. We show a strategy of computing the EDdegree of the hypothesis variety and determine its dimension. In Section \ref{sec:4} we give results on prescribed classes of varieties and numerical computation of the algebraic complexity of their hypothesis varieties. For the case of spheres, we provide a closed formula for the algebraic complexity in the $n=1$ case and conjecture a generalization of the formula for $n>1$ supported by numerical results. We end the section with establishing a connection between the studied classes and architectures of polynomial neural networks.

\subsection{Acknowledgments}
This work was partly made during the authors visit to ICERM during the fall of 2018 as part of the semester program on Non-linear algebra. The author would like to thank Sandra Di Rocco and David Eklund who have been of enormous help during the course of this project, Kathlén Kohn for suggesting to look at the EDdegree, and Bernd Sturmfels and Paul Breiding for useful discussions in the beginning of the project. 

\section{Preliminaries}\label{sec:2}
\begin{point}\label{pt point configuration}
	A \textbf{point configuration} $x= (x_1, \dots, x_m)\in \Bbb C^{mn}$ is a configuration of $m$ points $x_1, x_2, \dots, x_m$ in $\Bbb C^n$. When appropriate, we will interpret $x$ as a $(m\times n)-$matrix where each point, $x_i$, is a row in the matrix.
\end{point}
\begin{point}\label{pt veronese}
	The \textbf{Veronese map of degree $d$} is the map $v_d\colon \Bbb P^n \to \text{Sym}^d\, \Bbb P^n$ sending $x\in \Bbb P^n$ to its $d$th symmetric power. The \textbf{affine Veronese map of degree $d$} is the dehomogenization of $v_d$ in the sense that it sends $x\in \Bbb C^n$ to all possible monomials of degree $\leq d$, which is a vector in $\Bbb C^{{n+d} \choose d}$.
\end{point}
\begin{point}
	The \textbf{Euclidean distance function} is defined as $d_E(u, v) = \sum_{i=1}^n (u_i-v_i)^2$ where $u$ and $v$ are points in either $\Bbb R^n$ or $\Bbb C^n$. Let $V\subseteq \Bbb C^n$ be a variety and $p$ be a point in $\Bbb C^n$. A \textbf{critical point} of the Euclidean distance function with respect to $p$ and $V$ is a point $q\in V$ such that the line from $p$ to $q$ lies in the normal space of $V$ at $q$. 
\end{point}

\subsection{Multivariate Polynomial Interpolation}
Let $x = (x_1, \dots, x_m)\in \Bbb C^{mn}$ be a point configuration of $m$ points in $\Bbb C^n$ and assume that each point $x_1, \dots, x_m$ is sampled from a variety $V\subseteq \Bbb R^n$ without noise. It is then straightforward to proceed as in \cite{Breiding2018} to numerically estimate the coefficients of a set of polynomials defining $V$. The main tool for doing this is the \textit{Vandermonde matrix}, defined as follows:

\begin{Def}
	The \textit{Vandermonde matrix of degree $d$} of a point configuration $x$ is defined as:
	$$ U_{\leq d}(x) := \begin{bmatrix}
	v_d(x_1) \\
	\vdots \\
	v_d(x_m)
	\end{bmatrix}$$
	where $v_d\colon \Bbb C^n \to \Bbb C^{{n+d \choose d}}$ is the affine Veronese map of degree $d$.
\end{Def} 

\begin{example}
	For $X = \begin{bmatrix}
	x_1\\
	x_2
	\end{bmatrix}$ and $m=n=d=2$ we have that 
	$$v_2 \colon \Bbb C^2 \to \Bbb C^6,\ \ \ v_2(x, y) = [x^2, xy, y^2, x, y, 1]$$
	$$U_{\leq 2}(X) = \begin{bmatrix}
	v_d(x_1) \\
	v_d(x_2)
	\end{bmatrix} = \begin{bmatrix}
	x_{11}^2&x_{11}x_{12}& x_{12}^2& x_{11}& x_{12}& 1 \\
	x_{21}^2&x_{21}x_{22}& x_{22}^2& x_{21}& x_{22}& 1
	\end{bmatrix}$$
	Note that the coefficients of any polynomial up to degree 2 that vanishes on both $x_1$ and $x_2$ lie in the kernel of $U_{\leq 2}(X)$. 
\end{example}

In fact, it is true in general that the coefficients of any polynomial, $f$, of degree $\leq d$, that vanishes on $x$, lie in $\ker{U_{\leq d}(x)}$. A generating set for an ideal cutting out $V$ can thus be obtained by a choice of basis for $\ker{U_{\leq d}(x)}$. The problem one now faces is dealing with the numerical errors associated with computing the kernel of the Vandermonde matrix. For this we refer to \cite{Breiding2018}, where the authors explore ways of numerically estimating a generating set for the ideal cutting out $V$.

\subsection{Euclidean Distance Degree}\label{sec:edd}
The Euclidean Distance Degree (EDdegree) \cite{edd} counts the number of critical points of the squared Euclidean distance function from a generic point $p\in \Bbb C^n$ to a variety $V \subseteq \Bbb C^n$. The EDdegree thus counts the number of local minima, maxima and saddle points of this distance function. In this sense, the EDdegree is an algebraic complexity measure of polynomial optimization problems. It measures the degree of the ideal describing the critical points of the objective function subject to polynomial contraints. 

The most straighforward way of computing the EDdegree is by describing it as the degree of a certain ideal. Let 
$$I_{sing} = I + \langle c\times c\text{-minors of $J(f)$}\rangle$$
where $c$ is the codimension of $I$ and $J(f)$ the Jacobian matrix of the defining equations of $I = \langle f_1, \dots, f_s \rangle\subset \Bbb C[y_1, \dots, y_r]$. We assume that $I$ is the radical ideal of an irreducible variety for which we want to compute the EDdegree. The EDdegree is then equal to the degree of the following \textit{critical ideal}, which is defined as the saturation:
\begin{equation}\label{eq:eddideal}
\Bigg( I + \Bigg\langle (c+1)\times (c+1)\text{-minors of } \begin{pmatrix}
p-y\\
J(f)
\end{pmatrix} \Bigg\rangle \Bigg) \colon (I_{sing})^{\infty}
\end{equation}
By \cite[Theorem 2.7]{edd}, the critical ideal is always zero-dimensional if $p$ is a generic point in $\Bbb C^n$, and the points of the critical ideal are exactly the critical points of the squared Euclidean distance function from $p$ to $V$.

The following result gives a general upper bound of the EDdegree:
\begin{prop}[\cite{edd}]\label{prop:bound}
	Let $V$ be a variety of codimension $c$ in $\Bbb C^n$ that is cut out by polynomials $f_1, \dots, f_c, \dots, f_s$ of degrees $d_1\leq \dots \leq d_c\leq \dots \leq d_s$. Then
	$$ \text{EDdegree}(V) \leq d_1 d_2 \cdots d_c \sum_{i_1+\dots + i_c\leq n-c} (d_1-1)^{i_1}(d_2-1)^{i_2}\cdots (d_c-1)^{i_c}$$
	Equality holds when $V$ is a general complete intersection of codimension $c$. 
\end{prop}

\begin{example}(Eckart-Young Theorem)\label{ex:hyp}
	Let $p\in \Bbb C^{mn}$ be a configuration of $m$ points in $\Bbb C^n$ and suppose we want to approximate $p$ with a hyperplane in $\Bbb C^n$. The objective function that we optimize is the sum of the squared Euclidean distance function from each point in the configuration to its closest point on the hyperplane. Note that this is the Fréchet-norm of the matrix representation of $p-x$ where $x$ is a configuration lying on the hyperplane. It is a consequence of the Eckart-Young Theorem that the \textit{critical} hyperplanes can be computed analytically using the \textit{singular value decomposition (SVD)}. They are computed by first centering the configuration around the origin and then computing the SVD of $p$ represented as a matrix. Suppose $m\geq n$ and $$p=W_1\text{diag}(\sigma_1, \dots, \sigma_n)W_2$$ is the SVD of $p$ when represented as an $(m\times n)-$matrix. Then the critical hyperplanes are given by kernel elements of the matrices $$U_i = W_1\text{diag}(\sigma_1, \dots, \sigma_{i-1}, 0, \sigma_{i+1}, \dots, \sigma_n)W_2$$ where the $i$th singular value has been set to zero. Each $U_i \colon \Bbb C^n \to \Bbb C^m$ has a one-dimensional kernel and thus the number of critical hyperplanes equals $n$. This means that the EDdegree of the set of configurations lying on a hyperplane in $\Bbb C^n$ equals $n$ if $m\geq n$. The kernel elements of each $U_i$ are the \textit{principal components} of $p$ in the sense of \textit{principal component analysis (PCA)} \cite{mlbook}. 
%
%
%
%
\end{example}

\begin{remark}
	It is possible to replace the Euclidean distance function with a generic positive definite quadratic form, it is then called the \textit{generic} EDdegree. In the next section we will see what happens when we replace the Euclidean distance function with a pull-back of $d_E$ along a projection map, which results in a semi-definite quadratic form on the domain. 
\end{remark}

\section{The algebraic complexity of the Hypothesis variety}\label{sec:3}

In this section we define the optimization problem we want to analyze and the complexity measure used to study it, namely the algebraic complexity. The optimization problem is associated to a variety, called the hypothesis variety. We develop the tools we use to compute the algebraic complexity of this variety. 

Let $f_\beta \colon \Bbb C^n \to \Bbb C^r$ be a system of polynomials whose coefficients are polynomials in  $\beta\in \Bbb C^k$, and let $p=(p_1, \dots, p_m)\in \Bbb C^{mn}$ be a point configuration in $\Bbb C^n$. Let $V_\beta \subseteq \Bbb C^{n}$ denote the zero locus of the polynomial system $f_\beta$. Then $f_\beta$ define a class of varieties in $\Bbb C^n$, in fact it's a variety in $\Bbb C^n \times \Bbb C^k$. Our goal is to analyse the complexity of fitting a variety coming from the class $f_\beta$ to the point configuration $p$, in the sense of the following optimization problem:
\begin{equation}\label{eq:prob}
\begin{aligned}
&\min_{x\in \Bbb C^{mn}}&& \ \ \  \sum_{i=1}^m d_E(x_i, p_i)\\
&\ \ \ \ \text{s.t}&& \exists\, \beta\in \Bbb C^k \text{ s.t } f_\beta(x_i) = 0, \, \text{for all } 1\leq i\leq m
\end{aligned}
\end{equation}
The above objective function is the squared Euclidean distance function in the configuration space $\Bbb C^{mn}$. The optimization problem finds the closest configuration $x\in \Bbb C^{mn}$ to $p$ such that there is a variety $V_\beta \subset \Bbb C^n$ passing through each $x_i$. One may view the problem as being handed noisy samples $p$, sampled from a variety coming from the class $f_\beta$, and the goal is to recover the true values of $p$, which is done by finding the smallest perturbation of $p$ that lies on such a variety. Consider for instance the case of Example \ref{ex:hyp}, where $f_\beta$ is the class of hyperplanes in $\Bbb C^n$. Finding the minimal perturbation of $p$ is resolved by the Eckart-Young theorem, and computed using the SVD. The  singular values correspond to critical values of the objective function in Problem (\ref{eq:prob}), under the assumption that $p$ is centered around the origin. 
\begin{example}
	Fix a configuration of four points in the plane $\Bbb R^2$ centered around the origin, and let  $f_\beta = \beta_1 x_1 + \beta_2 x_2 + \alpha_1\beta_1 + \alpha_2\beta_2 + \alpha_3$ where $\alpha_1, \alpha_2, \alpha_3 \in \Bbb R$ are chosen randomly and $\beta \in \Bbb C^2$. We choose the $\alpha$'s in order to \textit{dehomogenize} the equation. The following figures illustrate the two lines passing through the two critical configurations of Problem (\ref{eq:prob}), which are both real and computed using the SVD:
	\begin{center}
		\begin{figure}[H]
		\begin{tabular}{llll}
			\includegraphics[width=0.3\linewidth]{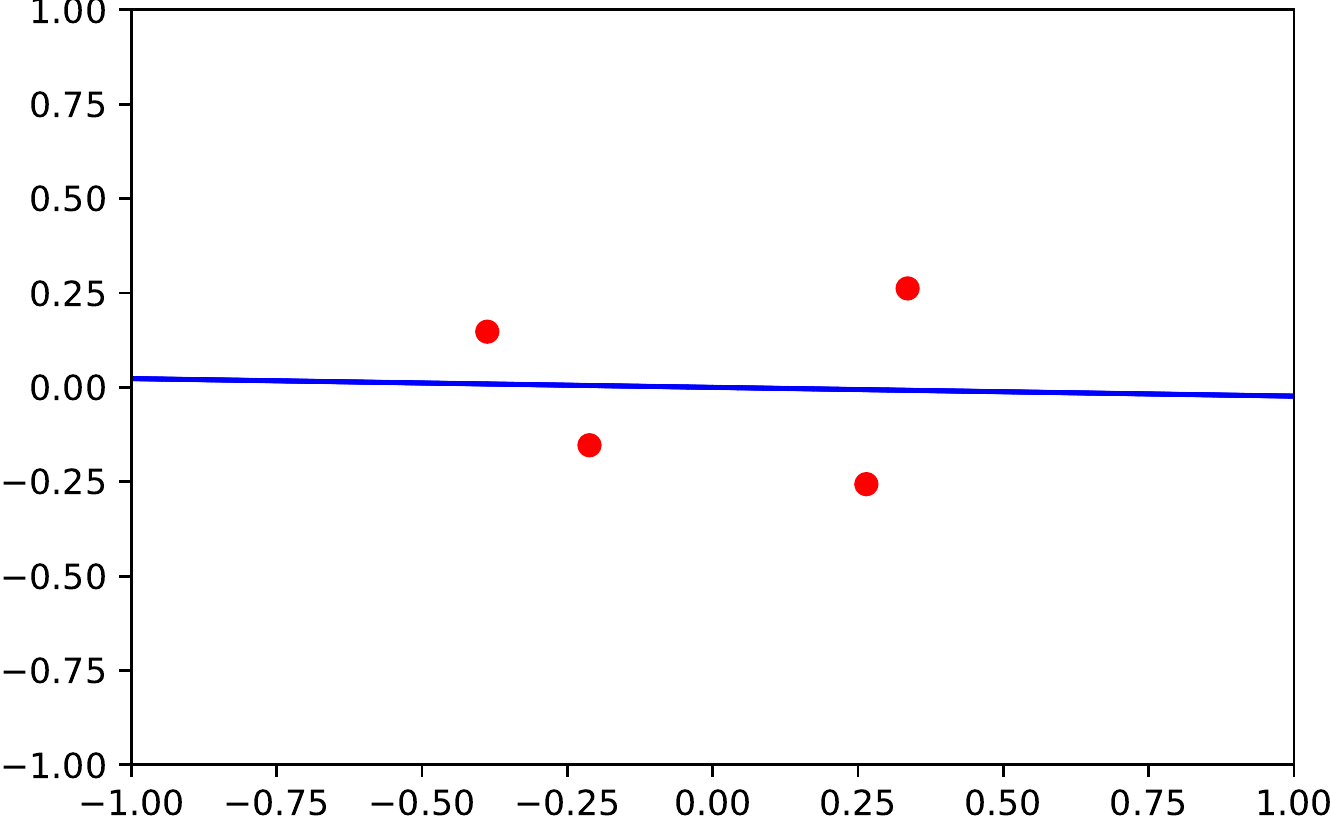}
			&
			\includegraphics[width=0.3\linewidth]{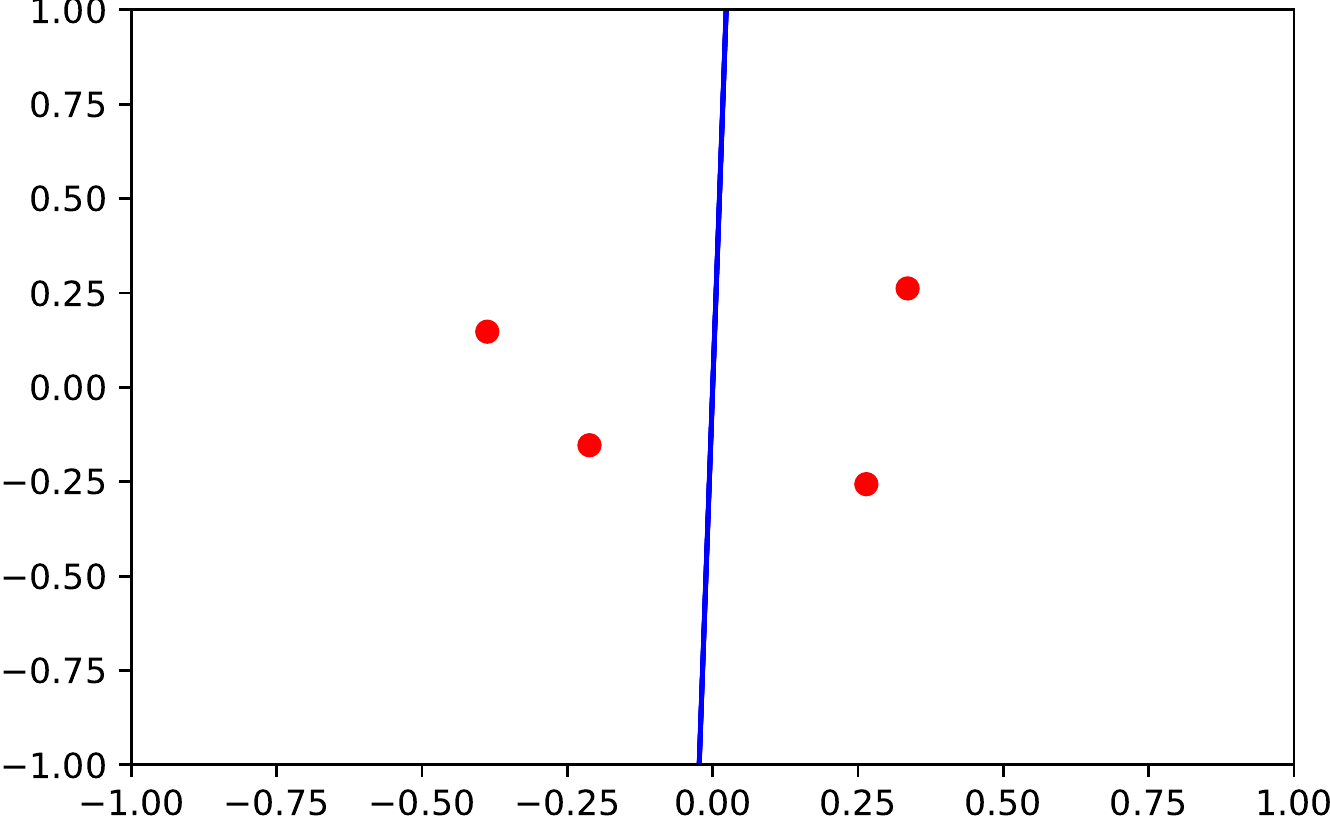}
		\end{tabular}
	\caption{The two critical lines to the point configuration.}
	\end{figure}
	\end{center}
As noted in Example \ref{ex:hyp} the two critical lines correspond to the principal components of the point configuration $p$. 
\end{example}
\begin{example}
	In Example \ref{ex:circle} we have the same configuration as in the previous example but we let  $f_\beta = (x_1 - \beta_1)^2 + (x_2-\beta_2)^2 - \beta_3$ and $\beta \in \Bbb C^3$. We then get critical circles as shown in the example.
\end{example}
The optimization problem (\ref{eq:prob}) is in general non-linear and has many local minima/maxima and saddle points. Our goal is to develop a complexity measure to study the complexity of Problem (\ref{eq:prob}) as $m$ and $n$ grows. We will study the complexity by considering the Euclidean Distance Degree of a certain variety, with respect to the point configuration $p$. The variety we will consider is the zero locus of the optimization problem (\ref{eq:prob}), by which we mean the set of configurations $p$ such that the global minimum of Problem (\ref{eq:prob}) is zero. This is the image of a variety under a projection whose closure the \textit{hypothesis variety}, $\HH_{m, n}$.
\begin{Def}\label{def:hypvar}
	Consider the incidence variety:
	$$C_{m, n} := \{ (x, \beta)\in \Bbb C^{mn}\times \Bbb C^k \mid f_\beta(x_i) = 0, \, 1\leq i\leq m \} \subseteq \Bbb C^{mn+k}$$  
	Define the \textbf{hypothesis variety}, denoted $\HH_{m, n}$, to be the algebraic closure of the image of $C_{m, n}$ under the projection $\pi\colon \Bbb C^{mn+k}\to \Bbb C^{mn}$ onto the first $mn$ coordinates.
	
Note that for a generic $x\in \HH_{m, n}$ there exists a $\beta\in \Bbb C^k$ such that $f_\beta(x_i) = 0$ for all $x_i$ in $x$. It is clear that for any $p$ in the zero locus of Problem (\ref{eq:prob}) it holds that $p\in \HH_{m, n}$. It follows from the above definition that $\HH_{m, n}$ is the algebraic closure of the zero locus of Problem (\ref{eq:prob}).

The algebraic complexity of finding the optimal solution of Problem (\ref{eq:prob}) may be characterized by the algebraic complexity of writing down the polynomial defining its solutions. The degree of this polynomial is the EDdegree of $\HH_{m, n}$. We refer to the function 
$$EDdegree(\HH_{-, -})\colon \Bbb N \times \Bbb N \to \Bbb N$$
as the \textbf{algebraic complexity} of $\HH_{m, n}$. 
\end{Def}
Computing the EDdegree of $\HH_{m, n}$ is however not straightforward since it is defined as the closure of the image of $C_{m, n}$ under the projection $\pi$. The defining equations of $\HH_{m, n}$ can be computed elimination ideals of $C_{m, n}$.  This computation is very costly since the result is a Gröbner basis for $\HH_{m, n}$, which is known to have doubly exponential computational complexity in the number of variables in the worst case. To remedy this we make the following definition.
 \begin{Def}
	Let $\beta EDdegree(C_{m, n})$ denote the EDdegree of $C_{m, n}$ using the Euclidean distance in $\Bbb C^{mn}$ pulled back to $\Bbb C^{mn+k}$ along $\pi$. This means that the distance between $p, q \in \Bbb C^{mn+k}$ is given by $d_E(\pi(p), \pi(q))$. Note that $d_E(\pi(p), \pi(q))$ may be zero even if $p\neq q$ and thus the pulled back distance is a pseudometric. Let $p\in \Bbb C^{mn}$, then the critical ideal is in this case defined as follows:
	\begin{equation}\label{eq:betaedd}
	\Bigg( I_{C_{m, n}} + \Bigg\langle (c+1)\times (c+1)\text{-minors of } \begin{pmatrix}
	(p, 0)-(x, 0)\\
	J(C_{m, n})
	\end{pmatrix} \Bigg\rangle \Bigg) \colon ((I_{C_{m, n}})_{sing})^{\infty}
	\end{equation}
	where $c$ is the codimension of $C_{m, n}$ in $\Bbb C^{mn+k}$ and  $I_{C_{m, n}}$ is the ideal of $C_{m, n}$ generated by its defining equations as given by Definition \ref{def:hypvar}. Finally, $J(C_{m, n})$ denotes the Jacobian matrix of $I_{C_{m, n}}$.  
\end{Def}
We might expect the EDdegree of $\HH_{m, n}$ to equal $\beta EDdegree(C_{m, n})$ but this is not the case, as shown by the following example.
\begin{example}
	Consider configurations of three points in the plane $\Bbb C^2$ and let  $\beta\in  \Bbb C^3$ and  $f_\beta = (x_1 - \beta_1)^2 + (x_2-\beta_2)^2 - \beta_3$. Then $\HH_{3, 2} = \Bbb C^6$ since any (distinct) three points in the plane have a unique circle passing through them, so $EDdegree(\HH_{3, 2})=1$. However, $\beta EDdegree(C_{3, 2})=4$.  Its critical points consists of the unique circle passing through the points and 3 additional circles illustrated by the following figures:
	\begin{center}
		\begin{tabular}{lll}
			\includegraphics[width=0.30\linewidth]{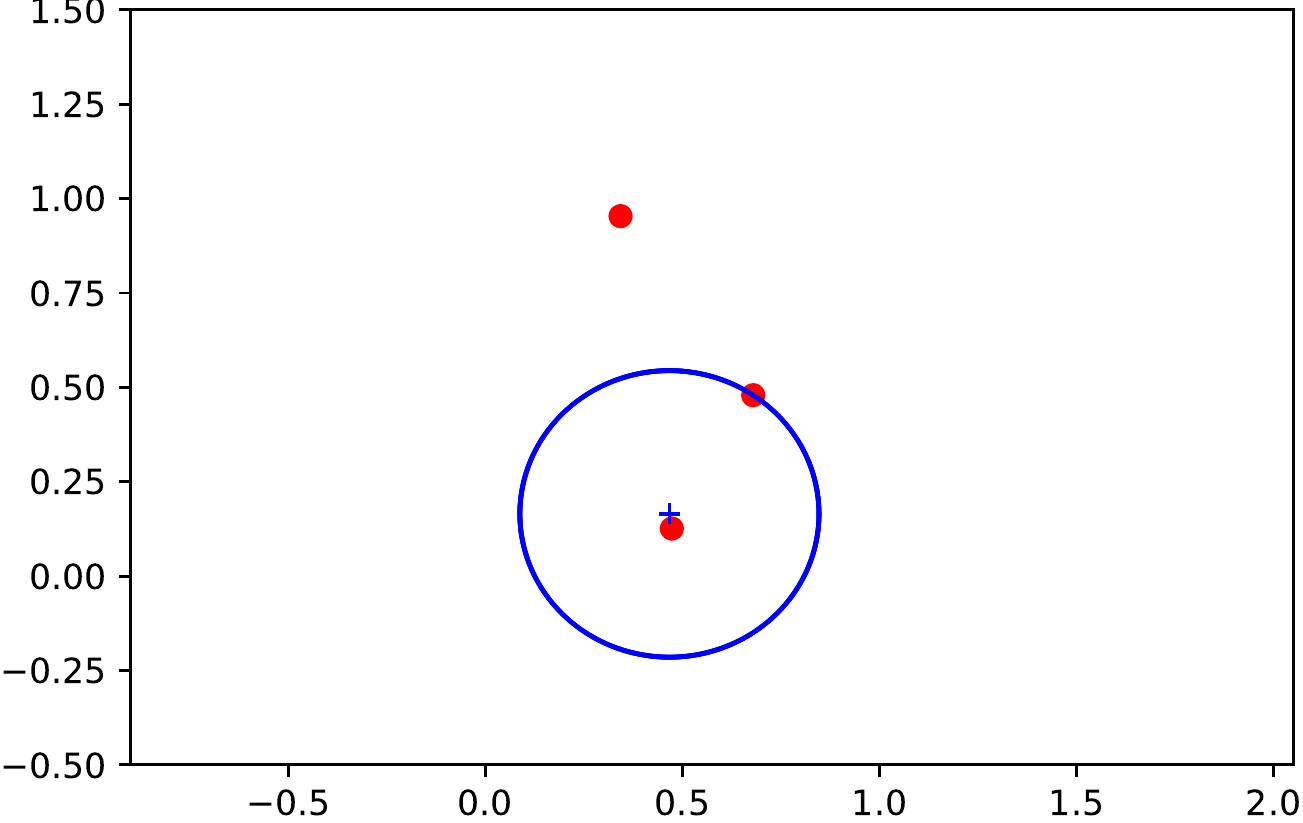}
			&
			\includegraphics[width=0.30\linewidth]{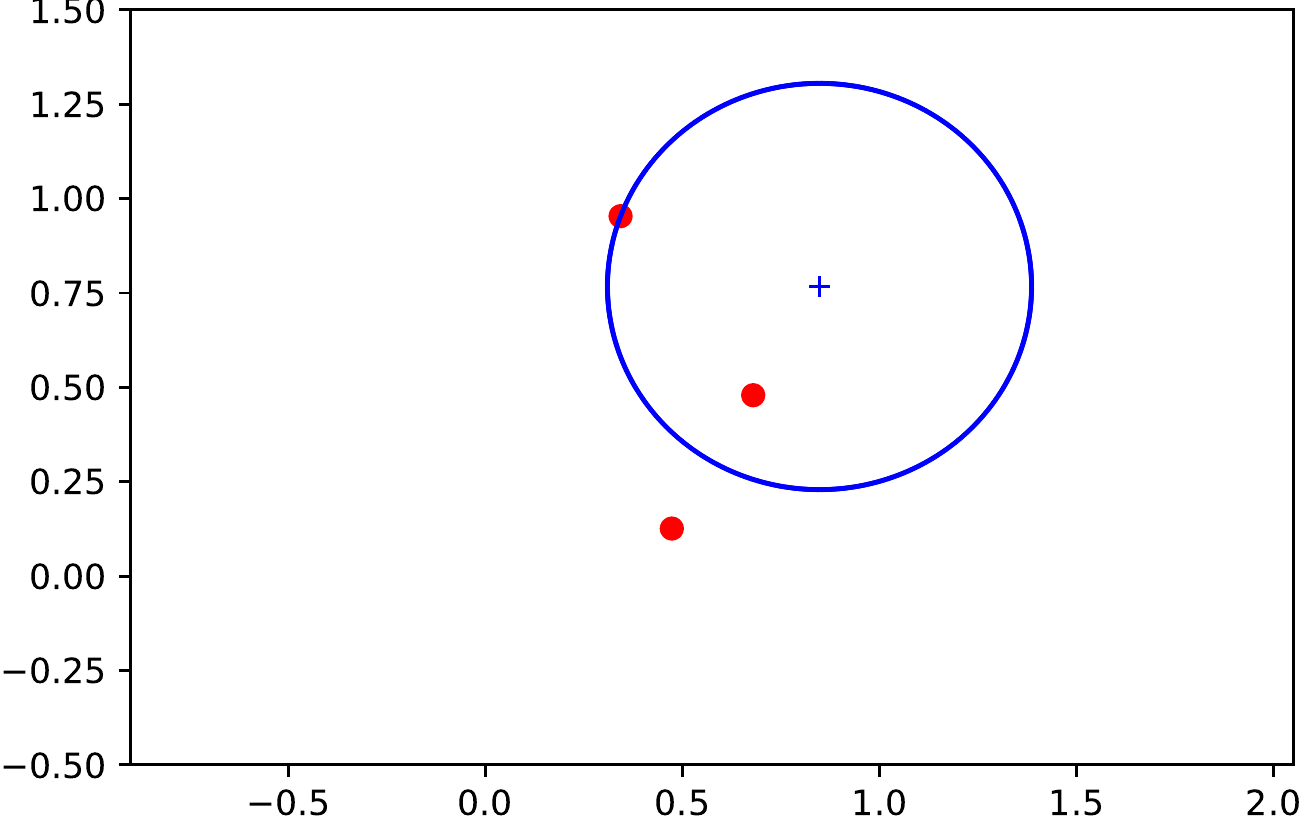}&	\includegraphics[width=0.30\linewidth]{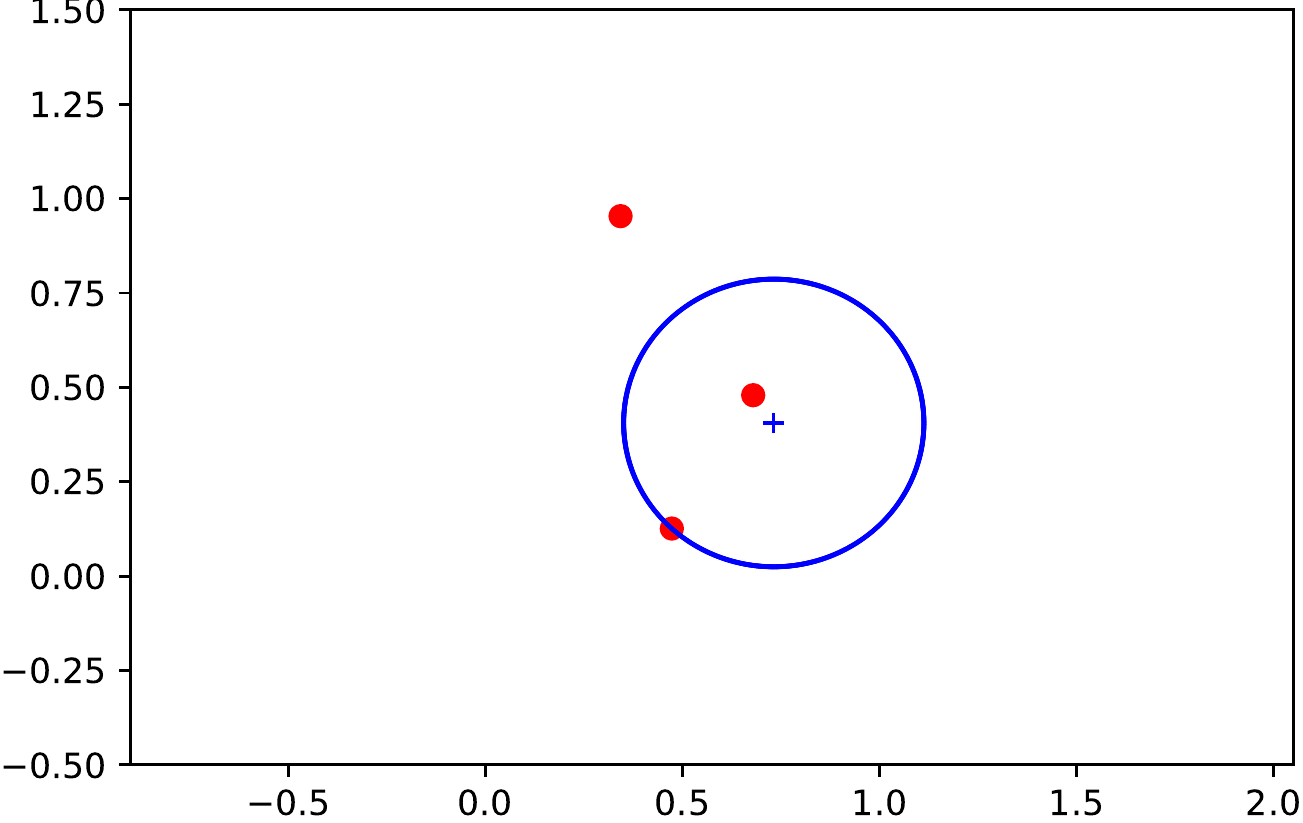}
		\end{tabular}
	\end{center}
	Note that in each critical configuration above, exactly two points are mapped to the same point on the circle. This would normally yield a singular point on $\HH_{m, n}$, but not in this case.
\end{example}
We are interested in the generic behavior of $C_{m, n}$ in the sense of generic choices of $\beta$ and generic fibers of $\pi$. What is interesting is thus the dominating components of $C_{m, n}$ under the projection $\tau\colon \Bbb C^{mn+k} \to \Bbb C^k$ that takes $(x, \beta) \to \beta$ and the dominating components of $C_{m, n}$ under $\pi$. We will throughout this section assume that there is a component dominating both $\pi$ and $\tau$. For convenience we assume that $C_{m, n}$ only has one component and is \textit{irreducible}. Note that if $f_\beta$ is irreducible for a generic $\beta\in\Bbb C^k$, then $C_{m, n}$ is irreducible, and thus it is reasonable to assume that $C_{m, n}$ is irreducible. 

The following result shows that $EDdegree(\HH_{m, n})$ is bounded from above by $\beta EDdegree(C_{m, n})$ and that the critical points of $\HH_{m, n}$ are a subset of the critical points of $C_{m, n}$. 
\begin{thm}\label{thm:ineq}
	$$EDdegree(\HH_{m, n}) \leq \beta EDdegree(C_{m, n})$$
	and suppose that $x\in \HH_{m, n}$ is critical to $\pi(p)$ for some $p\in C_{m, n}$, then any element in the fibre $\pi^{-1}(x)$ is critical to $C_{m, n}$ with respect to $p$. 
\end{thm}
\begin{proof}
	First note that by \cite[Lemma 2.1]{edd} it follows that $EDdegree(\HH_{m, n})$ is finite and thus if  $\beta EDdegree(C_{m, n})$ is infinite the inequality still holds. We may compute the image of the projection $\pi$ by computing a Gröbner basis $\{g_1, \dots, g_s\}$ of $C_{m, n}$ with an appropriate monomial ordering. A subset of this basis $\{ g_1, \dots, g_t \}$ describes the image and is a Gröbner basis for $\HH_{m, n}$. When we compute $\beta$EDdegree of $C_{m, n}$ we consider the following Jacobian matrix:
	$$\begin{bmatrix}
	0 & x_1-p_1 & \dots & x_m-p_m\\
	0 & \nabla_{x_1} g_1(x) & \dots &  \nabla_{x_m} g_1(x) \\
	\vdots & \vdots  & \dots &  \vdots  \\
	0 & \nabla_{x_1} g_t(x) & \dots &  \nabla_{x_m} g_t(x) \\
	\nabla_{\beta} g_{t+1}(x) & \nabla_{x_1} g_{t+1}(x) & \dots &  \nabla_{x_m} g_{t+1}(x) \\
	\vdots & \vdots  & \dots &  \vdots  \\
	\nabla_{\beta} g_{s}(x) & \nabla_{x_1} g_{s}(x) & \dots &  \nabla_{x_m} g_{s}(x) 
	\end{bmatrix} = \begin{bmatrix}
	0 & x-p\\
	0 & J(\HH_{m, n}) \\
	\nabla_{\beta} g_{t+1}(x) & \nabla_{x} g_{t+1}(x) \\
	\vdots  &  \vdots  \\
	\nabla_{\beta} g_{s}(x) & \nabla_{x} g_{s}(x)
	\end{bmatrix}	$$
	Note that the upper right block is the Jacobian matrix in the critical ideal of $\HH_{m, n}$. Consequently, any critical point of $\HH_{m, n}$ is also a critical point of $C_{m, n}$. It remains to prove two things: the first is that a critical point of $\HH_{m, n}$ does not yield a singular point of $C_{m, n}$. Generically, this would only happen if $C_{m, n}$ has a singular component but this cannot be since $C_{m, n}$ is assumed to be irreducible. The second thing is if a critical point of $\HH_{m, n}$ is not in the image  $\pi(C_{m, n})$, but generically this does not happen either since the image $\pi(C_{m, n})$ is constructible and thus contains a Zariski open subset of its closure $\HH_{m, n}$.
\end{proof}
\begin{remark}\label{rem:finite}
	Note that if $f_\beta$ is homogeneous in $\beta$, so that $f_{\lambda \beta}(x) = \lambda^s f_{\beta}(x)$ for some $s\in \Bbb N$. Then $\beta EDdegree(C_{m, n})$ is not finite since for any critical point $(x, \beta)$ we have that $(x, \lambda \beta)$ is also a critical point. Thus we have to assume that the map $\pi\colon C_{m, n} \to \HH_{m, n}$ is generically finite.
\end{remark}

Consider the variety of pairs $(x, p) \in \HH_{m, n} \times \Bbb C^{mn}$, denoted by $\mathcal{E}_{\HH}$, where $(x, p)$ is such that $x$ is a critical point of $\HH_{m, n}$ with respect to $p$. This is called the \textit{ED-correspondence} in \cite{edd}. Let $\tau_1 \colon \HH_{m, n} \times \Bbb C^{mn} \to \HH_{m, n}$ denote the projection onto the first component. Then $\text{dim}(\mathcal{E}_{\HH}) = \text{dim}(\HH_{m, n}) + \text{dim}(\tau_1^{-1}(x)) = mn$, where $x\in \HH_{m, n}$ is generic. If $\HH_{m, n}$ is irreducible, then $\mathcal{E}_{\HH}$ is an irreducible variety of dimension $mn$. This means that the projection $\tau_2 \colon \HH_{m, n} \times \Bbb C^{mn} \to  \Bbb C^{mn}$ has finite fibers, which is the same as saying that the critical ideal of $\HH_{m, n}$ with respect to $p$ is zero-dimensional (see the proof of \cite[Lemma 2.1]{edd}).  

We can construct an analogous object for $C_{m, n}$. Consider the variety $$\mathcal{E}_{C} := \{ (x, p) \in C_{m, n}\times \Bbb C^{mn} \mid (\pi(x)-p, 0)\in N_{x}(C_{m, n})\subseteq \Bbb C^{mn+k} \}$$
where $N_{x}(C_{m, n})$ denotes the normal space of $C_{m, n}$ at $x$. Assume that $\pi$ is generically finite, then $\text{dim}(C_{m, n}) =\text{dim}(\HH_{m, n})$. Let $\sigma_1\colon C_{m, n}\times \Bbb C^{mn} \to  C_{m, n}$ denote the projection onto the first component. From the proof of Theorem \ref{thm:ineq} we note that $\text{dim}(\sigma_1^{-1}(p)) \geq \text{dim}(\tau_1^{-1}(p))$. This then means that if $\pi$ is generically finite and $\text{dim}(\sigma_1^{-1}(p)) = \text{dim}(\tau_1^{-1}(p))$, then $\text{dim}(\mathcal{E}_{C}) = \text{dim}(C_{m, n}) + \text{dim}(\sigma_1^{-1}(p)) = mn$. Thus $\mathcal{E}_{C}$ is an irreducible variety of dimension $mn$, which means that the projection $\sigma_2$ onto the second component has finite fibers and thus that $\beta EDdegree(C_{m, n})$ is finite.

Note that if $\text{dim}(\sigma_1^{-1}(p)) > \text{dim}(\tau_1^{-1}(p))$ for a generic $p\in C_{m, n}$. Then this implies that $\text{rank} \nabla_x C_{m, n}(p) > \text{codim}(\HH_{m, n})$, which for dimensionality reasons forces $\text{rank} \nabla_\beta C_{m, n}(p) < k$. This means that the condition $\text{rank} \nabla_\beta C_{m, n}(p, \beta) = k$ is equivalent to $\text{dim}(\sigma_1^{-1}(p)) = \text{dim}(\tau_1^{-1}(p))$. We will now show that this condition is also equivalent to $\pi$ being generically finite.
%
%

\begin{lemma}\label{lem:finite}
	$\pi$ is generically finite if and only if $\text{rank}(\nabla_\beta C_{m, n}(x, \beta)) = k$ for a generic $(x, \beta)\in  C_{m, n}$.
\end{lemma}
\begin{proof}
	Suppose we fix a generic $x\in \HH_{m, n}$ and consider the system $f_\beta(x_1) = \dots = f_\beta(x_m) = 0$ of $m$ equations in $k$ variables. The condition that $\text{rank}(\nabla_\beta C_{m, n}(x, \beta)) = k$ is equivalent to saying that the associated variety is zero-dimensional, since we assumed that $C_{m, n}$ is irreducible. Conversely, if $\text{rank}(\nabla_\beta C_{m, n}(x, \beta)) < k$ the system does not have full rank which means that the variety describing the solutions is not zero-dimensional, which implies that $\pi$ is not generically finite. If $\pi$ is generically finite then the variety describing the solutions to the system is zero-dimensional, which implies that the system has full rank and thus that $\text{rank}(\nabla_\beta C_{m, n}(x, \beta)) = k$.
%
%
	%
\end{proof}

\begin{prop}\label{prop:finite}
	The critical ideal corresponding to $\beta EDdegree(C_{m, n})$ is zero-dimensional for a generic $p\in \Bbb C^{mn}$ if and only if $\pi$ is generically finite.
\end{prop}
\begin{proof}
	Suppose that $\beta EDdegree(C_{m, n})$ is zero-dimensional for a generic $p\in \Bbb C^{mn}$. By Theorem \ref{thm:ineq} this means that $\pi$ is generically finite and that $\text{dim}(C_{m, n}) =\text{dim}(\HH_{m, n})$. Now assume that $\text{rank}(\nabla_\beta C_{m, n}(p)) < k$. This forces $\text{rank} \nabla_x C_{m, n}(p) > m(n-1)+k$ and thus $\text{dim}(\sigma^{-1}(p)) > \text{dim}(\tau^{-1}(p))$. Now consider the projection $\rho\colon C_{m, n}\times \Bbb C^{mn} \to  \Bbb C^{mn}$ onto the second component. Since $\mathcal{E}_{C}$ is an irreducible variety of dimension $\geq mn$ it follows that the fibers of $\rho$ are not generically finite, which implies that the critical ideal of $C_{m, n}$ with respect to $p$ is not zero-dimensional. Thus it has to hold that $\text{rank}(\nabla_\beta C_{m, n}(p)) = k$ which by Lemma \ref{lem:finite} implies that $\pi$ is generically finite.
	
	Conversely, assume that $\pi$ is generically finite, by Lemma \ref{lem:finite} this implies rank$(\nabla_\beta C_{m, n}(x, \beta))= k$ for a generic $(x, \beta)\in \Bbb C^{mn}$. Then the differential $d\pi \colon TC_{m, n} \to T\HH_{m, n}$ is surjective at $(x, \beta)$. To see this, note that if $v\in T_{(x, \beta)}C_{m, n}$ and $d\pi(v)=0$, then the vector $v$ has only non-zero values for coordinates corresponding to $\beta$. The fact that $d\pi(x)$ is surjective implies that it is an isomorphsim, since dim$(C_{m, n}) = \text{dim}(\HH_{m, n})$, and thus any $v$ in the normal space of $C_{m, n}$ at $(x, \beta)$ is in the normal space of $\HH_{m, n}$ at $x$ under the projection $\pi$. The result then follows from the fact the the critical ideal of $\HH_{m, n}$ with respect to $x$ is zero-dimensional.
\end{proof}
\begin{cor}\label{prop:cond}
	Let $(x, \beta) \in \Bbb C^{mn+k}$ be a critical configuration of $C_{m, n}$ with respect to some $p\in \Bbb C^{mn+k}$. Then $x$ is a critical configuration of $\HH_{m, n}$ with respect to $\pi(p)$ if $\nabla_\beta C_{m, n}(x, \beta)$ has full rank.
\end{cor}
\begin{proof}
	As noted in the proof of Theorem \ref{prop:finite}; if $\nabla_\beta C_{m, n}(x, \beta)$ has full rank, then for any $p$ in the normal space of $C_{m, n}$ it holds that $\pi(p)$ is in the normal space of $\HH_{m, n}$ at $x$, which means that it is critical.
\end{proof}
Note that if $m=k$, then the normal space of $\HH_{m, n}$ at any $p\in \Bbb C^{mn}$ is a point, since $\HH_{m, n}=\Bbb C^{mn}$. This means that the only critical point of $\HH_{m, n}$ with respect to $p$ is $p$ itself. For $C_{m, n}$ however we have a different situation. Note that in this case $\nabla_\beta\, C_{m, n}(x, \beta)$ is a square matrix. Thus any critical point $(x, \beta)$ of $C_{m, n}$ with respect to $(p, 0)$, which is not $(p, 0)$ itself, would have  $\text{rank}\, \nabla_\beta C_{m, n}(x, \beta) < k$. In Example \ref{ex:circle} these critical points all come from the subvariety of $C_{m, n}$ consisting of points $(x, \beta)$ where $x$ is such that $x_i=x_j$ for some $1\leq i < j\leq m$. This means that all critical points for which $x\neq p$  lie on this subvariety of $C_{m, n}$ and are degenerate in this way. This also means that any $(x, \beta) \in C_{m, n}$ lie in the normal space of this subvariety on $C_{m, n}$. 

For $m>k$ we have observed through numerical experiments and proved for two special cases (see Section \ref{sec:spheres} and \ref{sec:linear}) that this does in general not happen. Therefore we state the following conjecture:
\begin{conj}\label{conj:eq}
	If $mr>k$ and $\pi$ is generically finite of degree $\omega$, then
	$$EDdegree(\HH_{m, n}) =\frac{1}{\omega} \beta EDdegree(C_{m, n})$$
\end{conj}

\subsection{Complete intersection}
In this section we assume that $C_{m, n}$ is a complete intersection. From the first statement below we will see that it suffices to assume that the zero locus of $f_\beta$ is a complete intersection for a generic choice of $\beta\in \Bbb C^k$. This allows us to read out the dimension of $\HH_{m, n}$ when the projection $\pi$ is generically finite. 

The fact that $C_{m, n}$ is a complete intersection is computationally advantageous since it allows using Lagrange multipliers instead of computing the minors in the critical ideal (see Section \ref{sec:2}). We end the section by showing that $\beta EDdegree(C_{m, n})$ is monotone in $n$ for a special case. 
\begin{prop}\label{prop:ci}
	Suppose that the zero locus of $f_\beta$ is a complete intersection for a generic $\beta\in \Bbb C^k$. Then $C_{m, n}$ is a complete intersection.
\end{prop}
\begin{proof}
	By definition, the ideal of $C_{m, n}$ is generated by $mr$ equations $f_\beta(x_1)=\dots = f_\beta(x_m)=0$. To show that $C_{m, n}$ is a complete intersection we will show that its codimension in $\Bbb C^{mn+k}$ is $mr$.
	To show that $C_{m, n}$ has codimension $mr$ we consider the rank of its Jacobian matrix:
	$$J(C_{m, n}) = \begin{bmatrix}
	\nabla_{\beta} f_\beta(x_1) & \nabla_{x_1} f_\beta(x_1) & 0 & \dots & 0\\
	\nabla_{\beta} f_\beta(x_2) & 0& \nabla_{x_2} f_\beta(x_2) & \dots & 0\\
	\vdots & \vdots & \vdots & \dots & \vdots \\
	\nabla_{\beta} f_\beta(x_m) & 0& 0 & \dots &\nabla_{x_m} f_\beta(x_m)
	\end{bmatrix}$$
	Note that the right part of the above matrix is block-diagonal and is of rank $mr$ if rank$(\nabla_{x_i} f_\beta(x_i))=r$ for all $1\leq i\leq m$ for any smooth point $(x, \beta)$ on $C_{m, n}$. Note that since $C_{m, n}$ is irreducible, this condition is satisfied by the fact that $f_\beta$ is a complete intersection for a generic $\beta$ and that if $(x_i, \beta)$ is a singular point of $f_\beta$, then $(x, \beta)$ is a singular point of $C_{m, n}$. Consequently, $C_{m, n}$ has codimension $mr$ and is thus a complete intersection.
\end{proof}

\begin{cor}
	If the conditions of Proposition \ref{prop:ci} hold, then codim$(\HH_{m, n}) = mr-k$.
\end{cor}
When $C_{m, n}$ is a complete intersection we can utilize a trick for computing the $\beta EDdegree(C_{m, n})$ more efficiently. In this case, the Jacobian $J(C_{m, n})$ has full rank and we can replace the minors in the critical ideal (\ref{eq:betaedd}) with Lagrange multipliers. This yields the following critical ideal:
\begin{cor}\label{cor:crit}
	Suppose $C_{m, n}$ is a complete intersection, then the critical ideal for computing $\beta EDdegree(C_{m, n})$ is given by:
	$$\bigg( I_{C_{m, n}} + \bigg\langle \sum_{i=1}^m \lambda_i \nabla f_\beta(x_i) - ((p, 0)-(x,0)) \bigg\rangle \bigg) \colon ((I_{C_{m, n}})_{sing})^{\infty}$$
\end{cor}

\begin{lemma}\label{lem:count}
	Suppose that $C_{m, n}$ is a complete intersection and has $s$ non-singular critical points for some specific point configuration $p'\in \Bbb C^{mn}$. Then $s\leq \beta EDdegree(C_{m, n})$.
\end{lemma}
\begin{proof}
	To show this we will construct a square system of equations describing the critical ideal, parametrized by a start configuration $p\in \Bbb C^{mn}$. The claim then follows from \cite[Theorem 7.1.1(2)]{bertini}. First note that $C_{m, n}$ is cut out by $mr$ equations in $mn+k$ variables. Since $C_{m, n}$ is assumed to be a complete intersection we may use the critical ideal of Corollary \ref{cor:crit}. This critical ideal uses Lagrange multipliers and thus introduces $mr$ new variables. There are $mn+k$ equations coming from the Jacobian $J(C_{m, n})$. In total this yields $mn+k+mr$ equations in $mn+k+mr$ variables, which thus yields a square system of polynomial equations.  This system is parametrized by the start configuration $p$ and can thus be described by a polynomial map $F(x, \beta; p)\colon \Bbb C^{mn+k+mr} \times \Bbb C^{mn} \to \Bbb C^{mn+k+mr}$. Let $N(p)$ denote the number of non-singular solutions of the system $F(x, \beta;p) = 0$ for a specific choice of $p$. Then by \cite[Theorem 7.1.1(2)]{bertini} it follows that $N(p') \leq N(p)$ for a generic choice of $p$, by which the statement follows.
\end{proof}

\begin{thm}\label{thm:inc}
	Suppose $f_\beta$ is such that $f_\beta(y) = \sum_{i=1}^n h(\beta_i)(y_i-\beta_{n+i})^2 + \beta_0$. Then,
	$$\beta EDdegree(C_{m, n}) \leq \beta EDdegree(C_{m, n+1})$$
\end{thm}
\begin{proof}
	Let $i\colon \Bbb C^n \hookrightarrow \Bbb C^{n+1}$ be the map that takes $$(y_1, \dots, y_n) \mapsto (y_1, \dots, 2^{-1/2}y_n, 2^{-1/2}y_n)$$ and let $\phi  \colon \Bbb C^{mn+2n+1} \hookrightarrow \Bbb C^{m(n+1)+2(n+1)+1}$ be the map that takes $$(x_1, \dots, x_m, \beta_0, \dots, \beta_{2n}) \mapsto (i(x_1), \dots, i(x_m), \beta_0, \dots, \beta_n, \beta_n, \beta_{n+1}, \dots, 2^{-1/2}\beta_{2n}, 2^{-1/2}\beta_{2n})$$ 
	Let $\overline{f_\beta}$ denote the corresponding polynomial for $C_{m, n+1}$. Note that if $(x, \beta)\in C_{m, n}$, then $\phi(x, \beta) \in C_{m, n+1}$ since:
	$$\overline{f_\beta}(i(y)) = \sum_{i=1}^{n-1} h(\beta_i)(y_i- \beta_{n+i})^2 + \frac{1}{2}h(\beta_n)(y_n- \beta_{2n})^2 + \frac{1}{2}h(\beta_n)(y_n- \beta_{2n})^2+ \beta_0 = f_\beta(y) = 0$$ 
	Suppose that $(x, \beta)$ is a critical point of $C_{m, n}$ with respect to a point configuration $p\in \Bbb C^{mn}$. We will show that $\phi(x, \beta)$ is a critical point of $C_{m, n+1}$ with respect to $\phi(p, 0)$. We do this by showing that $\phi(p, 0)$ is in the row space of $J(C_{m, n+1})|_{\phi(x, \beta)}$. $J(C_{m, n+1})|_{\phi(x, \beta)}$ is a $m\times( m(n+1)+2(n+1)+1)-$matrix which means that it has $m+2$ more columns than $J(C_{m, n})$. We will show that the added columns are just duplicates of columns from $J(C_{m, n})$ and that they do not add any equation to the critical ideal. 
	
	Note that the column of $J(C_{m, n})$ corresponding to $\beta_{n}$ is duplicated and the column corresponding to $\beta_{2n}$ is duplicated and scaled in $J(C_{m, n+1})|_{\phi(x, \beta)}$. The scaling is fine since its supposed to sum up to zero when multiplied by $\lambda_1, \dots, \lambda_m$ in the critical ideal. The other columns which differ are the columns corresponding to $x_{i, n}$ for $1\leq i \leq m$. These columns are also duplicated and scaled by $2^{-1/2}$. Note that the derivative of $f_\beta$ with respect to $x_{i, n}$ is linear in $x_{i, n}$ and $\beta_{2n}$. Also note that the $x_{i, n}$-coordinate of $\phi(p, 0)$ is also scaled by $2^{-1/2}$. This means that the whole equation in the critical ideal corresponding to this column is scaled by $2^{-1/2}$, which means that its solutions are the same. We can view it as if the column of $\begin{pmatrix}
	(x, 0)-(p, 0)\\
	J(C_{m, n})
	\end{pmatrix}$ corresponding to $x_{i, n}$ is mapped to to the following columns in $\begin{pmatrix}
	\phi(x, 0)-\phi(p, 0)\\
	J(C_{m, n+1})|_{\phi(x, \beta)}
	\end{pmatrix}$:
	$$\begin{bmatrix}
	x_{i, n} - p_{i, n} \\
	0 \\
	2h(\beta_i)(x_{i, n} - \beta_{2n}) \\
	0
	\end{bmatrix} \overset{\phi}{\mapsto} \begin{bmatrix}
	2^{-1/2} x_{i, n} - 2^{-1/2} p_{i, n} &  2^{-1/2} x_{i, n} - 2^{-1/2} p_{i, n}\\
	0  & 0\\
	2h(\beta_i)(2^{-1/2}x_{i, n} - 2^{-1/2}\beta_{2n}) & 2h(\beta_i)(2^{-1/2}x_{i, n} - 2^{-1/2}\beta_{2n})\\
	0 & 0
	\end{bmatrix}$$
	Thus we may conclude that $\phi(x, \beta)$ is a critical point of $C_{m, n+1}$ with respect to $\phi(p, 0)$. This means that any critical point of $C_{m, n}$ with respect to $(p, 0)$ is a critical point of $C_{m, n+1}$ with respect to $\phi(p, 0)$ under $\phi$. The statement then follows from Lemma \ref{lem:count}.
	%
	%
	%
	%
	%
	%
	%
\end{proof}

\begin{cor}
	Suppose $f_\beta$ is a system of equations of the form 
	\[
	\begin{aligned}
	\sum_{i=1}^n h_1(\beta_{1,i})(y_i- \beta_{1, n+i})^2  + \beta_{1,0}  &= 0\\
	\sum_{i=1}^n h_2(\beta_{2, i})(y_i- \beta_{2, n+i})^2  + \beta_{2,0}  &= 0\\
	\vdots& \\
	\sum_{i=1}^n h_r(\beta_{r, i})(y_i- \beta_{r, n+i})^2  + \beta_{r,0}  &= 0
	\end{aligned}
	\]
	such that for a generic choice of $\beta\in \Bbb C^k$, the zero locus of $f_\beta$ is a complete intersection. Then,
	$$\beta EDdegree(C_{m, n}) \leq \beta EDdegree(C_{m, n+1})$$
\end{cor}
\begin{proof}
	We use the same map $i\colon \Bbb C^n \hookrightarrow \Bbb C^{n+1}$ as in the proof of Theorem \ref{thm:inc}. Let $\phi  \colon \Bbb C^{mn+r(2n+1)} \hookrightarrow \Bbb C^{m(n+1)+r(2(n+1)+1)}$ be the map that takes $$(x_1, \dots, x_m, \beta_{1, -}, \dots, \beta_{r, -}) \mapsto (i(x_1), \dots, i(x_m), \psi(\beta_{1, -}), \dots, \psi(\beta_{r, -}))$$
	where $\beta_{j, -}$ denotes the tuple $(\beta_{j, 0}, \dots, \beta_{j, 2n})$ and $\psi \colon \Bbb C^{2n+1} \to \Bbb C^{2(n+1)+1}$ takes
	$$(\beta_0, \dots, \beta_{2n}) \mapsto (\beta_0, \dots, \beta_n, \beta_n, \beta_{n+1}, \dots, 2^{-1/2}\beta_{2n}, 2^{-1/2}\beta_{2n})$$
	Since $f_\beta$ is a complete intersection for a generic $\beta\in \Bbb C^k$ it follows from Proposition \ref{prop:ci} that $C_{m, n}$ is a complete intersection. The argument is then analogous to the proof of Theorem \ref{thm:inc}  and the statement follows from Lemma \ref{lem:count}.
\end{proof}

\section{Results for prescribed classes}\label{sec:4}
In this section we investigate the algebraic complexity of a number of classes of varieties. 

\subsection{Linear $\beta$}\label{sec:linear} Suppose that each polynomial in the system $f_\beta \colon \Bbb C^n \to \Bbb C^r$ has coefficients which are linear in $\beta$. Without loss of generality we may assume that $f_\beta$ is the system:
\begin{equation}\label{eq:linearsys}
\begin{aligned}
	\beta_{11} x^{\textbf{a}_{11}} + \beta_{12} x^{\textbf{a}_{12}} +\dots \beta_{1k_1} x^{\textbf{a}_{1k_1}}  &= 0\\
	\beta_{21} x^{\textbf{a}_{21}} + \beta_{22} x^{\textbf{a}_{22}} +\dots \beta_{2k_2} x^{\textbf{a}_{2k_2}} &= 0\\
	\vdots& \\
	\beta_{r1} x^{\textbf{a}_{r1}} + \beta_{r2} x^{\textbf{a}_{r2}} +\dots \beta_{rk_r} x^{\textbf{a}_{rk_r}} &= 0
\end{aligned}
\end{equation}

where $\textbf{a}_{11}, \dots, \textbf{a}_{rk_r} \in \Bbb N^n$ and $\beta \in \Bbb C^{\sum_{i=1}^r k_i}$. Note that for the projection $\pi\colon C_{m, n}\to \HH_{m, n}$ to be generically finite $f_\beta$ cannot be homogeneous in $\beta$. Therefore we need to let $\beta_{ik_i} = \sum_{j=1}^{k_i-1} \alpha_{ij}\beta_{ji}+\gamma_i$ for some generic $\alpha_{ij}, \gamma_i \in \Bbb C$, i.e. we choose a generic affine patch to de-homogenize the equations. 

Let $\textbf{a} := \{ \textbf{a}_1, \dots, \textbf{a}_k \}$ and define $U_{\textbf{a}}(x)$ to be the $m\times k$ submatrix of the Vandermonde matrix consisting of the columns corresponding to the monomials $x^{\textbf{a}_i}$ for each $1\leq i\leq k$. Note that if $f_\beta$ is the class of hyperplanes in $\Bbb C^n$, then $U_{\textbf{a}}$ equals the complete Vandermonde matrix of degree one. 
\begin{prop}\label{prop:princ}
	If $r=1$, then the hypothesis variety of $f_\beta$ is cut out by the principal minors of $U_{\textbf{a}}$. 
\end{prop}
\begin{proof}
	Suppose that $x\in \HH_{m, n}$, then there exists a $\beta$ such that $f_\beta(x_i) = 0$ for all $1\leq i\leq m$. This means that $\beta \in \ker{U_{\textbf{a}}(x)}$ and consequently that the principal minors of $U_{\textbf{a}}(x)$ vanish. 
	
	Conversely, suppose that the principal minors of $U_{\textbf{a}}(x)$ vanish. Then the kernel of $U_{\textbf{a}}(x)$ is non-trivial and there exists coefficients $\beta$ such that $f_\beta(x) = 0$, which means that $x\in \HH_{m, n}$.
\end{proof}
\begin{cor}
	If $r>1$, then the hypothesis variety of $f_\beta$ is cut out by the intersection of the  principal minors of $U_{\textbf{a}_1}, \dots , U_{\textbf{a}_r}$, where $\textbf{a}_i := \{\textbf{a}_{i1}, \dots, \textbf{a}_{ik_i}\}$.
\end{cor}
\begin{proof}
	By Proposition \ref{prop:princ} it follows that the system of the $i$th equations of $f_\beta$ for all $x_1, \dots, x_m$ is cut out by the principal minors of $U_{\textbf{a}_i}$. Consequently, the total system is cut out by the intersection of the principal minors of $U_{\textbf{a}_1}, \dots , U_{\textbf{a}_r}$.
\end{proof}
If $\textbf{a}_{1i} = \textbf{a}_{2i} = \dots = \textbf{a}_{ri}$ the above corollary implies the following result:
\begin{cor}\label{cor:minors}
	If $r>1$ and $k=k_1 = \dots = k_r$ and $\textbf{a}_{1i} = \textbf{a}_{2i} = \dots = \textbf{a}_{ri}$ for all $1\leq i \leq k$. Then the hypothesis variety of $f_\beta$ is cut out by the $(k-r+1)\times (k-r+1)-$minors of $U_{\textbf{a}}$.
\end{cor}
Let $f_\beta$ be the class of codimension $\leq r$ affine linear subspaces in $\Bbb C^n$. Note that $\pi$ is generically one-to-one. We can then show that Conjecture \ref{conj:eq} holds in this case:
\begin{thm}\label{thm:gene}
	If $m>n$, then
	$$EDdegree(\HH_{m, n}) = \beta EDdegree(C_{m, n})$$
\end{thm}
\begin{proof}
	We will use Proposition \ref{prop:cond} and show that the differential $d\pi$ is surjective on all critical points for a generic configuration $p\in \Bbb C^{mn+k}$. Recall that $V_\beta\subseteq \Bbb C^n$ is the variety of $f_\beta$ for a specific $\beta$. For $x$ to be critical to $p$ we must have that $p_i$ lie in the normal space of $x_i$ on $V_\beta$ for some fix $\beta\in \Bbb C^k$ for all $1\leq i \leq m$. Consider the $\beta$-part of the Jacobian matrix of $C_{m, n}$. If $\text{rank}\, \nabla_\beta C_{m, n}(x, \beta) < k$, then $x$ must lie on the intersection of $V_\beta$ with a hyperplane, which is a codimension $r+1$ affine linear subspace. Note that the normal space of this intersection on $V_\beta$ is a hyperplane in $\Bbb C^n$. We must thus have that $p$ lie on this normal space, which is a  hyperplane. Since $p$ is generic this cannot happen unless any choice of $p$ lie on a hyperplane, which happens when $m\leq n$. 
\end{proof}
In this case it holds that $U_{\textbf{a}_i} = U_1$ for $1\leq i\leq r$. Thus by Corollary \ref{cor:minors} it follows that the EDdegree of the hypothesis variety of $f_\beta$ is obtained by the following generalization of the Eckart-Young Theorem: 
\begin{thm}[\cite{edd}]
	If $m>n$, then
$$EDdegree(\HH_{m, n}) = {n \choose r}$$
\end{thm}
Thus for codimension $r$ affine linear subspaces we have a closed formula for the algebraic complexity.  

%
%
%
%


\subsection{Spheres} \label{sec:spheres}
Let $f_\beta$ be the polynomial
$$f_\beta(x) = \sum_{i=1}^n (x_i-\beta_i)^2 - \beta_0$$
Then $f_\beta$ is the class of all $(n-1)-$spheres in $\Bbb C^n$. The projection $\pi\colon C_{m, n} \to \HH_{m, n}$ is generically finite if $m\geq n+1$, in fact, it is generically one-to-one. Note that if we would replace $\beta_0$ with $\beta_0^2$ the map would be generically two-to-one. This would simply just double the number of critical configurations of $C_{m, n}$ to a configuration $p\in \Bbb C^{mn}$. 

Also in this case we can prove that Conjecture \ref{conj:eq} holds:
\begin{thm}\label{thm:eq}
	If $m>n+1$ then, $$EDdegree(\HH_{m, n}) = \beta EDdegree(C_{m, n})$$
\end{thm}
\begin{proof}
	We will use Proposition \ref{prop:cond} and show that the differential $d\pi$ is surjective on all critical points for a generic configuration $p\in \Bbb C^{mn+k}$. Consider the $\beta$-part of the Jacobian matrix of $C_{m, n}$:
	$$\nabla_\beta C_{m, n} = \begin{bmatrix}
	2(x_{11}-\beta_1) & \dots & 2(x_{1n}-\beta_n) & -1  \\
	2(x_{21}-\beta_1) & \dots & 2(x_{2n}-\beta_n) & -1 \\
	\vdots & \dots & \vdots & \vdots\\
	2(x_{m1}-\beta_1) & \dots & 2(x_{mn}-\beta_n) & -1 
	\end{bmatrix}
	$$
	Note that for this matrix to have a non-trivial kernel, we must have that $x_1, \dots, x_m$ lie on a hyperplane in $\Bbb C^n$. This means that $x_1, \dots, x_m$ lie on the intersection of a $(n-1)-$sphere and a hyperplane, i.e. an $(n-2)-$sphere. For $x$ to be critical to $p$ we must have that $p_i$ lie in the normal space of $x_i$ on the sphere. Note that the normal space of this $(n-2)-$sphere on $S$ is a cone in $\Bbb C^n$ and we must have that $p$ lie on this cone. For any $n$ points in $\Bbb C^{n+1}$ there is a $(n-2)$-sphere passing through them and then the last point determines the direction of the cone. Thus for $m>n+1$ there is no cone passing through a generic $p$, since the cone is fully determined by $n+1$ points. Consequently, any critical $x_1, \dots, x_m$ do not  lie on a hyperplane for a generic $p$ for $m>n+1$ and thus the result follows.
\end{proof}
By the above theorem we note that we can replace $C_{m, n}$ by $\HH_{m, n}$ in Theorem \ref{thm:inc}. We thus have the following corollary:
\begin{cor}
	$$EDdegree(\HH_{m, n}) \leq EDdegree(\HH_{m, n+1})$$
\end{cor}
For $n=1$ we can prove that there is a closed formula for the EDdegree of the hypothesis variety:
\begin{customthm}{1}
	$$EDdegree(\HH_{m, 1}) = 2^{m-1}-1$$
	\textit{and if $p$ is a real configuration, then the critical points of $\HH_{m, 1}$ with respect to $p$ are all real.}
\end{customthm}
\begin{proof}
	If $m=2$, then obviously $EDdegree(\HH_{m, 1})=1$. For $m>2$ we will use Theorem \ref{thm:eq} to compute the EDdegree of $\HH_{m, n}$ by computing the EDdegree of $C_{m, n}$. For $n=1$ we have the following defining equations for the critical ideal of $C_{m, 1}$:
	\begin{align*}
	(x_1-\beta_1)^2 - \beta_2 &= 0\\
	\vdots \\
	(x_m-\beta_1)^2 - \beta_2 &= 0\\
	2\lambda_1(x_1-\beta_1) &= x_1-p_1\\
	\vdots \\
	2\lambda_m(x_m-\beta_1) &= x_m-p_m\\
	-\sum 2\lambda_i(x_i-\beta_1) &= 0\\
	-\sum \lambda_i &= 0
	\end{align*}
	We may solve for each $x_i$ to get $x_i = \beta_1 + s_i \sqrt{\beta_2}$, where $s_i\in \{1, -1\}$ and the value of $\sqrt{\beta_2}$ is chosen such that the real part is non-negative. Note that when $s_i = s_j$ for all $1\leq i, j \leq m$ the above system only has a solution if $\beta_2 =0$ and $p_1=\dots =p_m$, which is not a generic configuration for $p$. Consequently, we will assume that  $s_i\neq s_j$ for some $i$ and $j$ and that $\beta_2\neq 0$. 
	
	By eliminating each $x_i$ and $\beta_1$, the above system of equations yields the  following matrix equation:
	$$
	\begin{bmatrix}
	-s_1 & s_2 & 0 & \dots & 0  \\
	-s_1 & 0 & s_3 & \dots & 0  \\
	\vdots & \vdots & \ddots & \dots & \vdots \\
	-s_1 & 0 & 0 & \dots & s_m \\
	s_1 & s_2 & s_3 &\dots & s_m \\
	1 & 1 & 1 & \dots & 1 
	\end{bmatrix}\lambda  = \begin{bmatrix}
	\frac{p_1-p_2}{\sqrt{\beta_2}} - (s_1-s_2) \\
	\frac{p_1-p_3}{\sqrt{\beta_2}} - (s_1-s_3) \\
	\vdots\\
	\frac{p_1-p_m}{\sqrt{\beta_2}} - (s_1-s_m)\\
	0\\
	0
	\end{bmatrix}
	$$
	Note that this system has a solution whenever the determinant of the following matrix vanishes:
	$$
	A:=\begin{bmatrix}
	-s_1 & s_2 & 0 & \dots & 0 & \frac{p_1-p_2}{\sqrt{y_2}} - (s_1-s_2) \\
	-s_1 & 0 & s_3 & \dots & 0 & \frac{p_1-p_3}{\sqrt{y_2}} - (s_1-s_3) \\
	\vdots & \vdots & \ddots & \dots & \vdots & \vdots\\
	-s_1 & 0 & 0 & \dots & s_m &\frac{p_1-p_m}{\sqrt{y_2}} - (s_1-s_m)\\
	s_1 & s_2 & s_3 &\dots & s_m & 0\\
	1 & 1 & 1 & \dots & 1 & 0
	\end{bmatrix}
	$$
	Let $S_i$ denote the $m\times m-$minor of the above matrix where the $i$th row and last column has been deleted. The determinant of $A$ can be expressed as:
	$$\text{det}(A) = \sum_{i=2}^{m} (-1)^{i}(\frac{p_1-p_i}{\sqrt{\beta_2}} - (s_1-s_i))S_{i-1}$$
	and the equation $\text{det}(A)=0$ has a solution whenever:
	$$\sum_{i=2}^{m} (-1)^{i}(p_1-p_i)S_{i-1} = \sqrt{\beta_2}\sum_{i=2}^{m}(-1)^{i}(s_1-s_i)S_{i-1}$$
	The $S_i$ will vanish only if $s_i=s_j$ for all $1\leq i, j\leq m$ which by assumption does not happen, otherwise they are non-zero. In fact, by symmetry $S_i = S_j$ for all $1\leq i, j\leq m-1$. Consequently, $\text{det}(A)=0$ if and only if:
	$$\sum_{i=2}^{m} (-1)^{i}(p_1-p_i) = \sqrt{\beta_2}\sum_{i=2}^{m}(-1)^{i}(s_1-s_i)$$
	which has a solution if and only if the signs on the sums on both sides are equal (the sign of $\sqrt{\beta_2}$ has already been determined). This will happen for half of all configurations of the $s_i$, i.e. for $\frac{2^m - 2}{2} = 2^{m-1}-1$ of them.
	
	Finally, note that if $p_1, \dots, p_m$ are real points then, $\beta_2$ is real, which will make $\lambda$ real and consequently yield $\beta_1$ and $x_1, \dots, x_m$ real.
\end{proof}
By Theorem \ref{thm:eq} it suffices to compute $\beta EDdegree(C_{m, n})$ in order to compute the EDdegree of $\HH_{m, n}$ and using the critical ideal of Corollary \ref{cor:crit} this is computationally tractable. The below table shows numerical results for some values of $m$ and $n$.
\begin{table}[H] \label{tab:exp}
\begin{center}
	\begin{tabular}{ c| c c c c c c c c c} 
		
		$n$\textbackslash $m$ & 2 & 3 & 4 & 5 & 6 & 7 & 8 & 9 & 10 \\ \hline
		1 & 1 & 3  & 7  & 15 & 31 & 63 & 127  & 255 & 511 \\
		2 & &4 & 26  & 90 & 266 & 730 & 1914 & 4858 & 12026 \\ 
		3 &  & & 31  & 127 & 423 & 1287 & 3703 & 10231  & 27383 \\ 
		4 &  & &   & 134 & 480 & 1564 & 4820 & 14260 & 40820 \\ 
		5 &  & &   &  & 489 & 1645 & 5265 & 16273 & 48881 \\

	\end{tabular}
\end{center}
\caption{Numerical results for computing $\beta EDdegree(C_{m, n})$.}
\end{table}
From the above table we make the following conjecture on the growth of the EDdegree as $m$ and $n$ grows.
\addtocounter{conj}{-2}
\begin{conj} 
The number of critical $(n-1)-$spheres of a generic configuration of $m$ points in $\Bbb C^n$, where $m>n+2$, is given by: 
	$$EDdegree(\HH_{m, n}) = EDdegree(\HH_{m-1, n}) +\sum_{k=0}^{m-1} { m-1 \choose k }\sum_{l=0}^n{{k+1} \choose l}- 2^{m-1} - m2^{m-2}$$
\end{conj}
For $n=1$ the above formula reduces to the one stated in Theorem \ref{thm:n1}.
\begin{remark}
	For $n=2$, the above sequence, as $m$ increases, is the cumulative sum of a known sequence, namely a binomial transformation of the (modified) triangle numbers (A084264).
\end{remark}

\subsection{Paraboloids}\label{sec:para}
Let $f_\beta$ be the polynomial
$$f_\beta(x) = (\sum_{i=1}^n \beta_i x_i)^2 + \sum_{i=1}^n \beta_{n+i}x_i + \beta_0$$
Then $\beta$ parametrizes all paraboloids in $\Bbb C^n$. Note that $f_\beta$ is not homogeneous in any of the $\beta$ variables and that the projection $\pi$ is generically finite. 

To compute the EDdegree of $\HH_{m, n}$ we have to compute a basis for the ideal corresponding to the Zariski closure of $\pi(C_{m, n})$. To do this we have to compute the $k$th order elimination ideal of the ideal of $C_{m, n}$ with the ordering $\beta, x$. This is intractable even for small examples. Therefore we use Bertini \cite{BHSW06} to compute $\beta EDdegree(C_{m, n})$ numerically. By Theorem \ref{thm:ineq} $\beta EDdegree(C_{m, n})$ serves as an upper bound on the EDdegree of $\HH_{m, n}$, and by Conjecture \ref{conj:eq} this bound is in this case believed to be an equality for $m>k$. 

The results of computing $\beta EDdegree(C_{m, n})$ for a series of $m$ and $n$ is shown in the table below. The computations for $n=2$ and $m>8$ and for $n>2$ either did not finish within five days of computation time or failed due to other errors related to the problem size.  
\begin{table}[H]
\begin{center}
	\begin{tabular}{ c| c c c c c c c c c} 
		
		$n$\textbackslash $m$ & 2 & 3 & 4 & 5 & 6 & 7 & 8 \\ \hline
		1 & 1 & 3  & 7  & 15 & 31 & 63 & 127   \\
		2 & & & 25  & 261 & 1219 & 4590 & 15722  \\

	\end{tabular}
\end{center}
\caption{Numerical results for computing $\beta EDdegree(C_{m, n})$.}
\end{table}
\subsection{Ellipsoids}\label{sec:ellipse}
Let $f_\beta$ be the polynomial 
$$f_\beta(x) = \sum_{i=1}^n \beta_{i}(x_i-\beta_{n+i})^2 - \beta_0$$
where $\beta_0 = \sum_{i=1}^n \alpha_i\beta_i + \gamma$ and $\alpha_i, \gamma \in \Bbb C$ are generic. We choose $\beta_0$ in such a way in order to (generically) de-homogenize the equation. Otherwise it would be homogeneous in the coordinates $\beta_0, \dots, \beta_n$ which would mean that $\pi$ is not generically finite. Then $f_\beta$ parametrizes  the set of all ellipsoids in $\Bbb C^n$.

As with the paraboloids in the previous section it is not tractable to compute the elimination ideals of $C_{m, n}$. Thus we may only upper bound the EDdegree of $\HH_{m, n}$ by computing $\beta EDdegree(C_{m, n})$, as by Theorem \ref{thm:ineq}. It holds in this case as well that by Conjecture \ref{conj:eq} we believe this upper bound to be an equality.

The results of computing $\beta EDdegree(C_{m, n})$ for a series of $m$ and $n$ is shown in the table below. As for the paraboloids, the computations for $n=2$ and $m>8$ and for $n>2$ either did not finish within five days of computation time or failed due to other errors related to the problem size. 
\begin{table}[H]
\begin{center}
	\begin{tabular}{ c| c c c c c c c c c} 
		
		$n$\textbackslash $m$ & 2 & 3 & 4 & 5 & 6 & 7 & 8  \\ \hline
		1 & 1 & 3  & 7  & 15 & 31 & 63 & 127   \\
		2 & & & 47  & 725 & 5217 & 28783 &  141507 \\

	\end{tabular}
\end{center}
\caption{Numerical results for computing $\beta EDdegree(C_{m, n})$.}
\end{table}
Note that the growth rate of $\beta EDdegree(C_{m, n})$ is increasing for the five examples above; the hyperplanes, affine subspaces of codimension $\leq r$, $(n-1)-$spheres, paraboloids and ellipsoids. This is what we would expect since each of the examples represent more and more complicated classes of varieties. 


\subsection{Connection to polynomial neural networks}\label{sec:nn}
A \textit{polynomial neural network} $N_\beta \colon \Bbb C^n \to \Bbb C^r$ is a alternating composition of linear and polynomial maps:
$$N_\beta (x) = f_l\circ \alpha_{l} \circ \dots  \circ f_1\circ \alpha_{1}$$
where $f_i \colon \Bbb C^{n_i} \to \Bbb C^{n_i}$ are polynomial maps, $\alpha_{i}\colon \Bbb C^{n_{i-1}} \to \Bbb C^{n_i}$ are affine linear maps and $n=n_0, n_1, \dots, n_l=r$ are the dimensions of the intermediate spaces. The composition $N_\beta$ is a polynomial map. A \textit{neural network architecture} is a specific choice of $n_0, n_1, \dots, n_l$ and $f_1, \dots, f_l$, leaving the linear maps $\alpha_1, \dots, \alpha_l$ free to be optimized over. Each neural network architecture thus defines a subclass (which is in fact a subvariety) of polynomial maps from $\Bbb C^n$ to $\Bbb C^r$. Properties of these type of classes of polynomial maps is studied in \cite{2019arXiv190512207K}.

We now note in the following examples that all of the prescribed classes studied in the previous sections can be expressed as polynomial neural network architectures. Consequently, we can characterize the complexity of their architectures using the algebraic complexity. 

\begin{example}[Linear]
	Let $\alpha_1\colon \Bbb C^n\to \Bbb C^r$ be the linear map described by Equation (\ref{eq:linearsys}) and let $f_1$ be the identity. 
\end{example}

\begin{example}[Spheres]
	Let $\alpha_1(z) = z-\beta \colon \Bbb C^n \to \Bbb C^n$, $f_1(z) = z\odot z$ and $\alpha_2(z) = \textbf{1}^T z - \beta_0 \colon \Bbb C^n \to \Bbb C$, where $\odot$ denotes the Hadamard product.  
	Then $$N_\beta(z) = \sum_{i=1}^n (z_i-\beta_i)^2 - \beta_0$$
	which is equivalent to a polynomial defining a sphere as in Section \ref{sec:spheres}.
\end{example} 
\begin{example}[Paraboloids] Let $\alpha_1(z) = Bz \colon \Bbb C^n \to \Bbb C^{2}$, $f_1(z) = (z_1^2, z_2)$ and $\alpha_2(z) = \textbf{1}^T z + \beta_0\colon \Bbb C^{2} \to \Bbb C$, where 
	$$B = \begin{bmatrix} 
	\beta_1 & \beta_2 & \dots & \beta_n \\
	\beta_{n+1} & \beta_{n+2} & \dots & \beta_{2n}
	\end{bmatrix}$$
	Then $$N_\beta(z) = (\sum_{i=1}^n \beta_i z_i)^2 + \sum_{i=1}^n \beta_{n+i}z_i + \beta_0$$
	which is equivalent to a polynomial defining a paraboloid as in Section \ref{sec:para}.
\end{example}
\begin{example}[Ellipsoids]
	Let $\alpha_1(z) = z-\beta \colon \Bbb C^n \to \Bbb C^n$, $f_1(z) = \text{diag}(c)z\odot z$ and $\alpha_2(z) = \textbf{1}^T z - \beta_0\colon \Bbb C^n \to \Bbb C$, where $\beta_0 = \sum_{i=1}^n \alpha_i\beta_i + \gamma$ as in Section \ref{sec:ellipse}.  
	Then $$N_\beta(z) = \sum_{i=1}^n c_i(z_i-\beta_i)^2 - \beta_0$$
	which is equivalent to a polynomial defining an ellipsoid as in Section \ref{sec:ellipse}.
\end{example}

\subsection{Computing the EDdegree}
All numerical results in the paper were computed on the Tegner supercomputer at Parallelldatorcentrum (PDC) KTH, using one or more computational nodes. Each node consists of 24 Intel E5-2690v3 Haswell cores and 512Gb of RAM. To counter the fact that the numerical methods may miss points in the critical ideal, all problem instances were run several times until no more new solutions appeared. 

\bibliographystyle{plain} 
\bibliography{interpolation}

\end{document}